\documentclass[a4paper,12pt,final]{article}
\usepackage{fullpage}
\usepackage{amsmath,amsthm}
  \usepackage{paralist}
  \usepackage{graphics} %% add this and next lines if pictures should be in esp format
  \usepackage{epsfig} %For pictures: screened artwork should be set up with an 85 or 100 line screen
\usepackage{graphicx}  \usepackage{epstopdf}%This is to transfer .eps figure to .pdf figure; please compile your paper using PDFLeTex or PDFTeXify.
 \usepackage[colorlinks=true]{hyperref}
\hypersetup{urlcolor=blue, citecolor=red}
\newtheorem{theorem}{Theorem}[section]

\newtheorem{proposition}{Proposition}

\theoremstyle{definition}
\newtheorem{definition}[theorem]{Definition}
\newtheorem{remark}{Remark}

\newcommand{\eps}[1]{{#1}_{\varepsilon}}

\title{Dynamics of charged elastic bodies\\ under diffusion at large strains}

\usepackage[normalem]{ulem}
\usepackage{ifthen}
\usepackage{color}
\usepackage{amssymb,amsmath,amssymb,amsthm}
\usepackage{hyperref} 
\usepackage{bm} 
\usepackage{todonotes}
\usepackage{framed}
\usepackage{mathrsfs}  
\usepackage{graphicx}
\usepackage{amsmath}
\usepackage{amssymb}
\usepackage{mathrsfs}
\usepackage{srcltx}
\usepackage{color}
\usepackage{hyperref}
\usepackage{upgreek}
\usepackage{psfrag}

\def\users{final-layout}  % when activatted, ``our'' debugging is suppressed

\ifthenelse{\equal{\users}{final-layout}}{
	\newcommand{\REPLACE}[2]{#2}
	
	\newcommand{\COMMENT}[1]{}
	\newcommand{\COMMENTGT}[1]{}
	\newcommand{\TODO}[1]{}
	\newcommand{\INTERNAL}[1]{}
	\newcommand{\QUESTION}[1]{}
	\newcommand{\DELETE}[1]{}

	\newcommand{\REM}[1]{\marginpar{\bfseries\tiny{\color{blue}}}}
}
{\newcommand{\REPLACE}[2]{{\color{brown}\sout{#1}\uline{#2}\color{black}}}
	
	\newcommand{\COMMENT}[1]{{\color{red}\uuline{#1}\color{black}}}
	\newcommand{\COMMENTGT}[1]{{\hfill\large\color{red}***{#1}***\color{black}\hfill}\\}
	\newcommand{\TODO}[1]{{\color{red}\uuline{#1}\color{black}}}
	\newcommand{\INTERNAL}[1]{\footnote{#1}}
	\newcommand{\QUESTION}[1]{{\color{brown}\uuline{#1}\color{black}}}
	\newcommand{\DELETE}[1]{{\color{brown}\sout{#1}\color{black}}}

	\newcommand{\REM}[1]{\marginpar{\bfseries\tiny{\color{blue}#1}}}
}

\newcommand\DELETEDELETE[1]{}

\usepackage{endnotes} %this is used to have notes at the end of the document with the command \endnote. Endnotes are shown using the command \theendnotes (from http://tex.stackexchange.com/questions/56145/is-there-a-way-to-move-all-footnotes-to-the-end-of-the-document)
\usepackage{pdfpages} %this is used to include notes at the end of the paper.

\definecolor{giuseppe}{RGB}{0,0,80}
\definecolor{gt}{RGB}{0,0,80}
\definecolor{hide}{RGB}{200,200,200}
\definecolor{inprogress}{RGB}{100,100,100}
\definecolor{issue}{RGB}{0,255,255}
\definecolor{tocheck}{RGB}{150,150,150}
\newcommand\comment[1]{}

\newcommand\dx{\mathop{{\rm d}x}}
\newcommand\dt{\mathop{{\rm d}t}}

\newcommand{\R}{\mathbb{R}}
\newcommand{\N}{\mathbb{N}}
\renewcommand{\d}{\mathrm{d}}

\newcommand\DT[1]{\mathchoice
	{{\buildrel{\hspace*{.1em}\text{\LARGE.}}\over{#1}}}
	{{\buildrel{\hspace*{.1em}\text{\Large.}}\over{#1}}}
	{{\buildrel{\hspace*{.1em}\text{\large.}}\over{#1}}}
	{{\buildrel{\hspace*{.1em}\text{\large.}}\over{#1}}}}
\newcommand\DDT[1]{\mathchoice
	{{\buildrel{\hspace*{.1em}\text{\Large.\hspace*{-.1em}.}}\over{#1}}}
	{{\buildrel{\hspace*{.1em}\text{\large.\hspace*{-.1em}.}}\over{#1}}}
	{{\buildrel{\hspace*{.1em}\text{\large.\hspace*{-.1em}.}}\over{#1}}}
	{{\buildrel{\hspace*{.1em}\text{\large.\hspace*{-.1em}.}}\over{#1}}}}
\newcommand{\Vdots}{\mathchoice{\:\begin{minipage}[c]{.1em}\vspace*{-.4em}$^{\vdots}$\end{minipage}\;}{\:\begin{minipage}[c]{.1em}\vspace*{-.4em}$^{\vdots}$\end{minipage}\;}{\:\tiny\vdots\:}{\:\tiny\vdots\:}}

\newcommand{\divS}{\mathrm{div}_{\scriptscriptstyle\textrm{\hspace*{-.1em}S}}^{}}
\newcommand{\Cof}{\operatorname{\rm{Cof}}}
\newcommand{\eq}{\eqref}
\renewcommand{\eps}{\varepsilon}

\renewcommand{\nu}{\sigma}

\newcommand{\stress}{\bm S}
\newcommand{\calE}{\mathscr{E}}
\newcommand{\calF}{\mathscr{F}}

\def\nablaIT{\setbox0=\hbox{$\nabla$}%
	\pdfliteral{q 1 0 .3 1 0 0 cm}\rlap{$\nabla$}\pdfliteral{Q}\kern\wd0 }
\newcommand\nablait{\operatorname{\nablaIT}{}}
\newcommand\dnablait{\nablait^2}
\newcommand{\DDD}[3]{\begin{array}[t]{c}#1\vspace*{-1em}\\_{#2}\vspace*{-.45em}\\_{#3}\end{array}}
\newcommand{\ddd}[3]{\DDD{\begin{array}[t]{c}\underbrace{#1}\vspace*{.6em}\end{array}}{\text{\footnotesize #2}}{\text{\footnotesize #3}}}
\begin{document}
\author{Tom\'a\v s Roub\'i\v cek\footnote{Mathematical Institute, Charles University,
 	Sokolovsk\'a 83, 18675~Praha~8, Czech Republic, and  
 	Institute of Thermomechanics, Czech Acad. Sci., 
 	Dolej\v skova 5, 18200 Praha 8, Czech Republic. Email: \texttt{tomas.roubicek@mff.cuni.cz}}~ and Giuseppe Tomassetti\footnote{Universit\`a degli Studi Roma Tre, Dipartimento di Ingegneria,
 	Via Vito Volterra 62, 00146~Roma, Italy. Email: \texttt{giuseppe.tomassetti@uniroma3.it}}}

\begin{sloppypar}	
  \maketitle
%  \end{sloppypar}
%\end{document}

\begin{abstract}
	We present a model for the dynamics of elastic or poroelastic bodies with 
	monopolar repulsive long-range (electrostatic) interactions at large strains. 
	Our model respects (only) locally the non-self-interpenetration condition 
	but can cope with 
	possible global self-interpenetration, yielding thus a certain
	justification of most of engineering calculations which ignore these
	effects in the analysis of elastic structures. These models 
	necessarily combines Lagrangian (material) description with Eulerian (actual)
	evolving configuration evolving in time. Dynamical problems 
	are studied by adopting the concept of nonlocal nonsimple materials,
	applying the change of variables formula for Lipschitz-continuous mappings,
	and relying on a positivity of determinant of
	deformation gradient thanks to a result by Healey and Kr\"omer.
\end{abstract}
\paragraph{2010 Mathematics Subject Classification.} 35Q60, % PDEs in connection with optics and electromagnetic theory
	35Q74, % PDEs in connection with mechanics of deformable solids
	65M60, % Finite elements, Rayleigh-Ritz and Galerkin methods, finite methods
	74A30, % Nonsimple materials.
	74F15, %Coupling of solid mechanics with other effects: Electromagnetic effects
	76S99, %Flows in porous media; filtration; seepage: None of the above, but in this section
	78A30. % Electro- and magnetostatics

% Please provide minimum  5 keywords.
\paragraph{Keywords. } {Elastodynamics, \color{black}poroelasticity\color{black},
	nonsimple materials, 
	long-range interactions.}
      \makeatletter
      \edef\@thefnmark{}
\@footnotetext{\paragraph{Acknowledgments.} {The authors are thankful
	to an anonymous referee for many comments to the model and to
	Dr.\ Giuseppe Zurlo for a discussion on the concept
	and applicability of the ideal dielectric model. \color{black}
	This research was partly supported through the grants
	17-04301S (as far as dissipative evolution concerns)
	and 19-04956S (as far as dynamic and nonlinear behaviour concerns)  of the
	Czech Science Foundation, 
	through the institutional project RVO:\,61388998 (\v CR),
	and the Grant of Excellence Departments, MIUR-Italy
	(Art.1, commi 314-337, Legge 232/2016),
	as well as through
	INdAM-GNFM.}  
    }
    \makeatother

%The title of your section 1

\section{Introduction}
In most applications of continuum mechanics, mechanical self interactions 
between two parts of the same body $\varOmega$ consist only 
in contact forces exchanged by these parts along their common boundary. There 
are, however, situations of physical interest 
such as for instance electrically-charged, self-gravitating, magnetized, or polarized bodies,
% self gravitating bodies, or electrostatically charged bodies, 
where mechanical interactions between 
parts of the same body are non negligible even if these parts are separated 
by a positive distance. 

\def\mono{m}
\def\calE{\mathscr{E}}
\def\MM{\bf m}
\def\mm{\bm m}
In these situations, a peculiarity of the ensuing mathematical model is that the equations that govern the evolution of the body are formulated in Lagrangean form, (i.e.\ the independent space variable belongs to a fixed reference configuration) while the 
equations that determine the long-distance self interactions \DELETEDELETE{through the potential of the electrostatic (likewise  gravitation,  demagnetising  or  depolarizing)  field} {(through the potential of the electrostatic, gravitation,  demagnetising  or  depolarizing  field)} are formulated in the Eulerian setting (i.e.\ the independent space variable belongs to the entire space). Combining the Lagrangian and Eulerian descriptions usually requires injectivity of the deformation related with non-selfpenetration, which can be ensured in static or quasistatic problems but seems extremely difficult in dynamical problems if it would be handled with a mathematical rigor. Overcoming this difficulty, most engineering calculations under large strains ignores non-selfpenetration, too.

To further clarify the issues related to a possible non-invertibility of the deformation, let us consider, as a prototype of body supporting long-range self interactions of monopolar type, for instance an electrically-charged body occupying in its reference configuration a regular region $\varOmega$ of $\R^d$, where $d$ is the space dimension. The most natural way to specify the charge content of the body (just like mass content) is through a \emph{referential charge density} $q:\varOmega\to\mathbb R$, whose value $q(x)$ at a referential point $x\in\varOmega$ is the amount of charge per unit volume at that point when the body is in its undeformed state.

When the body undergoes a deformation $\bm\chi:\varOmega\to\R^d$, the charge bound to the material points of the body undergoes a rearrangement. The energetic cost associated to such rearrangement is the integral over $\R^d$ of the squared gradient of a scalar potential determined by solving the Poisson equation with source term a \emph{spatial charge density} $\mathsf q:\R^d\to\mathbb R$ supported on $\bm\chi(\varOmega)$. If the deformation is smooth and invertible both locally and globally, in the sense that it is one-to-one and its Jacobian $J=\det\nabla\bm\chi$ is bounded from below by a positive constant, then the spatial charge density at a given point $\mathsf x\in\R^d$ is given by $\mathsf q(\mathsf x)=q(\bm\chi^{-1}(\mathsf x))/\det(\nabla\bm\chi(\bm\chi^{-1}(\mathsf x)))$ if $\mathsf x\in\bm\chi(\varOmega)$, and $\mathsf q(\mathsf x)=\mathsf 0$ otherwise. Since the standard formulation of the boundary-value problem that governs nonlinear elasticity does not ensure injectivity of the deformation map (the model does not incorporate the physics of self contact), we face an issue when willing to give sense to the notion of spatial charge density.

Similar issues are encountered in bodies supporting long-distance self interaction of dipolar type. Although the mechanical treatment of these types of interactions is even more subtle than those of monopolar type, as pointed out in \cite{DesimonePodioguidugli1997}, the relevant  mathematical literature appears richer. In the static and quasistatic cases, when equilibria are sought through energy minimization, non self-penetration can be ensured by enforcing the Ciarlet-Ne\v cas condition \cite{CiaNec87ISCN}. This fact was exploited, for example, in the paper \cite{Rogers1988}, which contains the analysis of the variational problem that governs equilibrium configurations of polarized and magnetized elastic bodies. Existence results are also offered in \cite{KrStZe15ERIM} in the incompressible case. An existence theorem for magnetized and polarized body can be found in \cite{Silhavy2017a}, along with a characterization of the conditions under which the stored energy is polyconvex.
%Still in the static case, the analysis of a problem in magnetoelasticity taking into account diffusion and temperature effects was carried out in \cite{RouTom2018}. 

In dynamical problems, on the other hand, the device of energy minimization is of no use so that, in particular, the mentioned Ciarlet-Ne\v cas cannot be exploited, and handling non-selfpenetration with mathematical rigor would necessarily require the addition of  further ingredients that take into account self contact. {For time-independent problems in second-gradient elasticity, an result from \cite{PalHea17ISCS} is available, showing that there exist weak solutions which
	satisfy the constraint of injectivity, provided that equilibria are attained with the help of a normal reaction in the form of a Radon measure supported on the boundary. In dynamics, however, no such result is available to our knowledge.}

%The issue seems to have been avoided in the mathematical literature so far and, still in the case of long distance dipolar interactions, these have been incorporated only in the case of small strains \cite{Roubicek2013}.\COMMENT{I WOULD DELETE THIS SENTENCE + OUR REFERENCE - SMALL STRAINS/DIPOLES NOT RELEVANT HERE(?)}

This issue of non-invertibility of the deformation map was already pointed out in a previous paper of ours \cite{RouTom2018} and prompted us for further investigations. In the present undertaking we propose a mathematical formulation {which, by means of Federer's change-of-variable formula \cite{Fede69GMT},  does not rely on the injectivity of the deformation map.}

Of course, because of the inherent features of classical elastodynamics whose 
particular hyperbolic structure prevents a proper mathematical treatment 
\cite{Ball2002}, we are anyhow to add further ingredients to conventional 
nonlinear elasticity. Precisely, we include a regularization which 
consists in a non-local energy that depends 
%\DELETEDELETE{nonlocally} 
on the second gradient of the deformation map. We refer to the resulting model 
as a non-local \DELETEDELETE{and} non-simple \REPLACE{body}{material}. 
Beside the important analytical regularizing property {for the otherwise 
	nonlinear hyperbolic system}, one motivation of such nonlocal nonsimple 
material-concept is dispersion of elastic waves with a large degrees of 
freedom {for} covering both normal and anomalous dispersion;
cf.\ the analysis in \cite{Jira04NTCM} performed at small strains. Further 
discussions on nonlocal theories can be found, for example, in \cite{Eringen2015}.\COMMENT{MAYBE STILL MORE CITATIONS???}

The singularity of the kernel of \DELETEDELETE{this}{the} nonlocal term is chosen to guarantee that deformations with finite energy belong to a suitable fractional Sobolev space. Such device has two benefits: first, by as a consequence of a result of T.\,Healey and S.\,Kroemer \cite{HeaKro09IWSS}, the determinant of the deformation gradient is bounded from below by a positive constant; second, since the deformation is Lipschitz continuous, it is possible to apply \DELETEDELETE{a standard}{the aforementioned} change-of-variable theorem, which is intimately related to the area formula \cite{Fede69GMT}, to provide an alternative formulation of the equations that govern the scalar potential that carries the information about long-distance interactions. This fact constitutes a novelty with respect to our previous treatment in \cite{RouTom2018}, where self interactions were not accounted for in the dynamic setting. {The choice of a fractional Sobolev space rather than an integer type permits us to use a quadratic 
	regularization of the stored energy, which is crucial to keep linearity of 
	the term that contains higher-order spatial derivatives in the hyperbolic-type 
	evolution equation. (In principle, one can also think about 
	%We have discarded the option of 
	a local quadratic regularization 
	involving the third-order gradient
	% because
	but this would have produced additional complications in the formulation of 
	the boundary conditions.)}

Although the formulation we propose applies to both monopolar and dipolar self interactions, in the present paper we limit our analysis to interactions of monopolar type for technical reasons. To be more specific, the term that accounts for monopolar long-distance interactions appears in the force balance in the form of the \emph{first gradient} of the scalar potential, composed with the deformation; for dipolar interactions, the corresponding term involves the \emph{second gradient}. In the first case, by a suitable manipulation, we can cast the long-distance force term in a form that does not involve the gradient of the scalar potential; in the second case, this is not possible.

A further limitation of our analysis is the assumption
that the material be an \emph{ideal dielectric} in the sense of
\cite{Zhao2007}. This assumption amount to neglecting the
material part of the electromechanical coupling. From the technical
standpoint, this limitation is somehow imposed by the mathematical
structure of the problem, which makes it difficult to obtain strong
convergence of the electric field. Indeed, if we wanted to include
electromechanical coupling, we should add to the free energy a term
that depends on the gradient of the electric field, in a manner
similar to what we do with the deformation gradient. On the other
hand, the ideal dielectric model, notwithstanding its limitations, finds several applications (see for
instance \cite{Zurlo2018} for an application to material
instabilities).\color{black}

A related technical issue also forces us to replace the conventional Poisson equation, with a $p$-regularization, so as to guarantee that in our constructive approach to existence of solutions the approximated scalar potentials converge in the space of continuous functions. {When taking into account the electrostatic scalar potential in the mechanical force balance, the potential must be pulled back into the reference configuration. Such pull back operation involves a composition between the scalar potential and the deformation map, which may be very badly if the scalar potential converge strongly only in a Lebesgue space. In fact, we need a stronger convergence in the space of continuous functions.}

We leave it as an open problem the application of our formulation within the standard setting of electrostatics, i.e., without the $p$-regularization.

%\DELETE{In conventional engineering calculations the aforementioned difficulties are overcome just by ignoring self-penetration. Here, reflecting this engineering approach, we want to develop a theory which does not rely on the injectivity of deformation and can admit selfpenetration, without discussing physical relevance of situations when the bodies indeed penetrate themselves.}

The plan of our paper is the following: Section~\ref{sec-mono}
exposes the repulsive monopolar long-range interactions  when charges are bound to the body, 
%of both attractive and repulsive character, 
and discuss the abstract structure of the problem.
%, and outlines some generalizations.
The rigorous analysis as far as an existence of weak 
solutions is then performed in Section~\ref{sec-mono-anal} by employing
rather constructive Galerkin approximation.
It is then generalized for time-evolving monopoles (electric charge) 
related to some diffusant (in this case, charge can rearrange themselves in the body). This is done in Section~\ref{sec-diffusion} 
%presents the modification for the 
%dipolar interactions which are in certain aspects similar to the
%repulsive monopolar interactions but possess specific differences. 
by using the concept of poroelastic solids and Biot-like diffusion driven
by the gradient of a chemical potential. 
\color{black} Yet, as pointed out already by M.A.\,Biot at al.
\cite{biot1972theory}, this poroelastic-like model ``is not
restricted to the presence of actual pores. The fluid may be in
solution in the solid, or may be adsorbed. Such pheonomena are
usually associated with the concept of capillarity or osmotic pressure.''
\color{black}
In the final Section~\ref{sec-rem}, we end the paper by outlining some 
modifications or generalizations and the difficulties which accompany them.
%\DELETE{We will distinguish the variables in the mentioned 
%reference versus the actual space configurations by fonts: 
%the reference-configuation notation will be italicised/slanted,
%while the space-configuation notation will be uprighted. Thus,
%e.g.\ we will distinquish the potential $\phi$ versus 
%$\upphi$, the domains $\varOmega$ versus $\Omega$, 
%or the gradients or $\nablait$ versus $\nabla$, etc.}

\section{Elastodynamics with monopolar repulsive interactions}\label{sec-mono}
%        ~~~~~~~~~~~~~~~~~~~~~~
%        \COMMENT{MOVE NOTATION LATER}

As a general \DELETEDELETE{typographic}{typographical} convention, we shall use italicised/slanted fonts to denote mathematical objects that pertain to the reference configuration, and upright fonts for objects that pertain to the current configuration. Consistent with this convention we distinguish the referential domain $\varOmega$ from its image $\Omega$ under the deformation map. We shall stick to a similar convention when dealing with differential operators: for example, we will use either $\nablait$ or  $\nabla$ when referring to fields whose domain of definition is the reference domain $\varOmega$ or the physical space $\R^d$.\COMMENT{we do not use $\Omega$ anymore, and it is not true that we use this notation systematically, since we write $\rm div$ in the same manner both in Eulerian and Lagrangian setting. I would completely remove this comment, which is not indispensable, and which may otherwise confuse some readers when we make exceptions to the above convention!}

Thorough the whole paper, the reference configuration $\varOmega$ is assumed to be a bounded Lipschitz domain in $\R^d$, whose boundary we denote by $\varGamma$. We use the standard notation for the Lebesgue $L^p$-spaces and
$W^{k,p}$ for Sobolev spaces whose $k$-th distributional derivatives 
are in $L^p$-spaces. We will also use the abbreviation $H^k=W^{k,2}$. 
%\tiny Moreover, we use the standard notation  $p'=p/(p{-}1)$, and 
%$p^*$ for the Sobolev exponent $p^*=pd/(d{-}p)$ for $p<d$ while
%$p^*<\infty$ for $p=d$ and $p^*=\infty$ for $p>d$,
%and the ``trace exponent'' $p^\sharp$ defined 
%as $p^\sharp=(pd{-}p)/(d{-}p)$ for $p<d$ while
%$p^\sharp<\infty$ for $p=d$ and $p^\sharp=\infty$ for $p>d$.
%Thus, e.g., $W^{1,p}({\varOmega})\subset L^{p^*}\!({\varOmega})$ or 
%$L^{{p^*}'}\!({\varOmega})\subset W^{1,p}({\varOmega})^*$=\,the dual to $W^{1,p}({\varOmega})$. \normalsize
In the vectorial case, we will write $L^p({\varOmega};\R^d)\cong L^p({\varOmega})^d$ 
and $W^{1,p}({\varOmega};\R^d)\cong W^{1,p}({\varOmega})^d$. \COMMENT{DO WE USE THIS NOTATION?} Also,
we admit $k$ noninteger with the reference to the Sobolev-Slobodetski\u\i\ 
spaces. Note that, in this notation, we have the compact embedding 
$H^{2+\gamma}({\varOmega})\subset W^{2,p}({\varOmega})$
if $p>2d/(d-2\gamma)$ and $W^{2,p}({\varOmega})\subset W^{1,p^*}({\varOmega})$.
In particular $H^{2+\gamma}({\varOmega})\subset C^1(\bar{\varOmega})$ if 
$d<p<2d/(d-2\gamma)$, which can be satisfied if $\gamma>d/2-1$ as
employed in \eqref{ass-HK} to facilitate usage of the results from 
\cite{HeaKro09IWSS}, cf.\ \eqref{det-positive-on-lev} below. We also denote by $\operatorname{meas}_d$ the 
$d$--dimensional Hausdorff measure. 
%\DELETE{Yet, we will confine ourselves to the
%physically relevant case $d=3$.}

On the time interval $I=[0,T]$, we consider the Bochner spaces 
$L^p(I;X)$ of Bochner measurable mappings $I\to X$ whose norm is in $L^p(I)$,
with $X$ being a Banach space, while $C_{\rm w}(I;X)$ will denote
the Banach space of weakly measurable mappings $I\to X$. When writing estimates, we denote by $C$ a generic positive 
constant which may change from one formula to another. 
%In case more than one such constant appears in the same formula, we shall decorate its symbol with a progressive numerical subscript (see for instance \eqref{} below) 

\color{black}When electrostatic interactions are
accounted for, the most general constitutive equation for the free
energy includes a dependence on both the
deformation and on the electric field (see for example the discussion
in Chap.4 of \cite{Dorfmann2014}, and in particular Eq. 4.38
therein). This general assumption would allow us to incorporate several coupling effects, such as
for instance classical piezoelectricity
\cite{Toupin1956,Toupin1963}. On the other hand, interesting
electromechanical effects can still be captured  through the \emph{ideal
	dielectric model}  \cite{Zhao2007}. In this model, the referential free energy density splits additively in a
mechanical part and an electrostatic part. More specifically, this is equivalent to letting the function
$\phi$ in \cite[Eq. 4.56]{Dorfmann2014} (see also Eq.s 4.37 and 4.38 of
the same reference) to depend only on the deformation. The
contribution of the electric part of the free energy will be
introduced later. \color{black}As to the mechanical part, we rely on a
non-simple material model by defining \color{black}the mechanical part
of the free--energy of the body as

\begin{align}
\calE_{\rm mech}(\bm\chi)=&
\int_\varOmega\!\varphi(\nablait \bm\chi)
\,\d x+\mathscr{H}(\dnablait\bm\chi),
\label{static-gravity-functional}
\end{align}
with $\varphi:{\rm GL}^+(d)\times\R\to\R$ a specific free energy with 
${\rm GL}^+(d):=\{\bm F\in\R^{d\times d};\ \det\bm F>0\}$.
The quadratic form $\mathscr{H}$ in \eqref{static-gravity-functional} is 
defined by 
\begin{align}
\mathscr{H}(\nablait^2\bm\chi)=\frac{1}{4}\sum_{i=1}^d\int_{\varOmega\times\varOmega}(\nablait^2 \chi_i(x){-}\nablait^2\chi_i(\tilde x))
%\Vdots
{:}\mathfrak{K}(x,\tilde x)
%\Vdots
{:}(\nablait^2\chi_i(x){-}\nablait^2\chi_i(\tilde x))
\,\d\tilde x\,\d x
\label{8-nonlocal-f-hyper}
\end{align}
with the hyperelastic-moduli 
symmetric positive-semidefinite kernel 
$\mathfrak{K}:\varOmega\times\varOmega\to\R^{d\times d\times d\times d}$ satisfying $\mathfrak{K}(x,\tilde x)=\mathfrak{K}(\tilde x,x)$
and with the scalar $\chi_i:\varOmega\to\R$ the $i$--th component of the deformation map. Thanks to Fubini's theorem, the non-local strain energy has the representation
\begin{align}
\mathscr{H}(\nablait^2\bm\chi)=\frac{1}{2}\sum_{i=1}^d\int_{\varOmega\times\varOmega}\mathfrak H_i(\nablait^2\chi_i):\nablait^2\chi_i\,\d x,
%\Vdots
\label{9-nonlocal-f-hyper}
\end{align}
where the second--order tensors $\mathfrak K_i$, $i=1,\ldots,d$ are defined as the G\^ateaux differential of $\mathscr{H}$ with respect to the $i$--th component $\bm\chi_i$ of the deformation, namely
\begin{equation}\label{eq:76}
\big[\mathfrak{H}_i(\nablait^2\chi_i)\big](x)=
\int_{\varOmega}
%(G_i(x)-G_i(\tilde x)){:}
\mathfrak{K}(x,\tilde x)
%\Vdots
{:}(\nablait^2\chi_i(x)-\nablait^2\chi_i(\tilde x))\,\d\tilde x.
\end{equation}
This construction automatically ensures frame indifference of the regularizing energy, and if we denote by $\{\bm e_i\}_{i=1}^d$ the canonical basis of $\R^d$, then the third--order tensor
\begin{equation}\label{eq:75}
\mathfrak H=\sum_{i=1}^d \bm e_i\otimes\mathfrak H_i
\end{equation}
is the hyperstress work conjugate of $\nablait^2\bm\chi$.

%\COMMENT{.... frame indifferency............}
%The triple dot 
%product on the right--hand side of the kernel $\mathfrak K$ is understood .............
As we shall see below with mode detail, the kernel $\mathfrak{K}$ of the regularization term $\mathscr H(\nablait^2\bm\chi)$ is chosen singular 
on the diagonal $\{x=\tilde x\}$
in such a way to ensure that deformations with bounded energy are in a fractional Sobolev space $H^{2+\gamma}(\varOmega;\R^d)$ with \DELETEDELETE{$\gamma\ge 1/2$}{$\gamma\ge d/2$}. This entails, in particular, that if the deformation $\bm\chi(t)$ at time $t$ has bounded energy, then $\bm\chi(t)\in W^{2,p}(\varOmega;\R^d)$ with $p>d$. Then, {as shown in \cite{HeaKro09IWSS}}, a suitable growth assumption on the bulk free energy $\varphi$ as the determinant of its deformation gradient tends to zero guarantees that the determinant is uniformly bounded from below away from zero, and hence in particular the deformation is locally invertible.
%As already pointed out in the introduction, within the static setting, besides local invertibility, we may also achieve global invertibility of solutions by imposing the Ciarlet-Ne\v cas condition. \COMMENT{WE ALREADY SAID IT} To the best of our knowledge, however, global invertibility of solutions cannot be guaranteed in the dynamical setting, not even with the aid of the regularizing term. 

Let us assume that the body is endowed with a Lagrangian density $q(x)$ of electric charge which is responsible for long-distance interaction, both between the body and the exterior and between the body with itself. If we assume that $q\in L^1(\varOmega)$, then $q$ induces a signed measure on $Q$ on the reference configuration in the standard fashion: for every Borel set $\mathcal P\subset\varOmega$, the quantity $Q(\mathcal P)=\int_{\mathcal P}q(x)\d x$ is the total charge contained in $\mathcal P$.

We stipulate that when the body is set in motion the charges are redistributed in space in a manner described by a set--valued function $\mathsf Q(t,P)$ defined in the following fashion: given a \emph{spatial control volume} $P \subset\R^d$, the amount of charge contained in $P$ after the deformation is 
$\mathsf Q(t,P):=Q({\stackrel{\leftarrow}{\bm\chi}}(t,P))
=\int_{\stackrel{\leftarrow}{\bm\chi}(t,P)}q(x)\,\d x$, where {we use the notation
	for the set-valued inverse deformation field:} 
$$
\stackrel{\leftarrow}{\bm\chi}(t,P)=\{x\in\varOmega:\bm\chi(t,x)\in P\}.
$$ 
In other words, the charge contained in $P$ is the amount of charge continued in the counterimage of $P$ under $\bm\chi(t)$.\COMMENT{Maybe $\varPi$ in place of $\mathcal P$ and $\Pi$ in place of $P$ here?}

Clearly, for this definition to make sense, we must be sure that the counterimage of $P$ is a Lebesgue--measurable set (this is another motivation for putting in the model an ingredient that enforces deformations maps to be regular). This is true, however, because if $\bm\chi\in W^{2,p}(\varOmega;\R^d)$ then $\bm\chi$ is Lipschitz continuous and hence it maps Borel sets into Borel sets and viceversa. Moreover, it can be shown that  $\mathsf Q(t,\cdot)$, the Eulerian charge distribution at time $t$, admits the following density:
\begin{equation}\label{eq:55}
\mathsf q(t,\mathsf x):=\sum_{x\in\stackrel{\leftarrow}{\bm\chi}(t,\mathsf x)}\frac{q(x)}{\det(\nablait\bm\chi(t,x))}\qquad\forall (t,\mathsf x)\in I\times\mathbb R^d.
\end{equation}
This result is indeed a consequence of a well-known change-of-variables formula due to Federer \cite{Fede69GMT} (see also  \cite[Thm.~3.9]{EvansGariepy2nd}), which says that for every integrable function $f$, 
\begin{equation}\label{eq:54}
\int_{\mathbb R^d}\sum_{x\in\stackrel{\leftarrow}{\bm\chi}(t,\mathsf x)}{f(x)}\,\d\mathsf x=\int_{\varOmega}f(x)\det(\nablait\bm\chi(t,x))\,\d x.
\end{equation}
Let us notice that, in the special case when $f=\mathsf f\circ\bm\chi$, with $\mathsf f$ a Lebesgue-measurable  function on $\Omega(t)=\bm\chi(t,\varOmega)$, formula \eqref{eq:54} becomes $\int_{\Omega(t)}\mathsf f\operatorname{card}\stackrel{\leftarrow}{\bm\chi}(t,\mathsf x)\,\d\mathsf x=\int_{\varOmega}\mathsf f\circ\bm\chi\det(\nablait\bm\chi(t,x))\,\d x$. This last formula was proved my Marcus and Mizel in \cite{Marcus1973} under the sole assumption $\bm\chi(t,\cdot)\in W^{1,p}(\varOmega;\R^d)$ (i.e., without requiring Lipschitz continuity of the deformation). Moreover, when \eqref {eq:54} is used with $f=1$, we obtain the Area Formula \cite{Fede69GMT}
\begin{equation}
\int_{\mathbb R^d} \operatorname{card}\Big(\stackrel{\leftarrow}{\bm\chi}\!(t,\mathsf x)\Big)\operatorname{\d\mathsf x}=\int_\varOmega \det(\nablait\bm\chi(t,x))\operatorname{\d\mbox{$x$}}<\infty,
\end{equation}
a result which implies that the cardinality of the 
preimage $\stackrel{\leftarrow}{\bm\chi}\!(t,\mathsf x)$ is finite for almost all $\mathsf x$, so that the summation in \eqref{eq:55} extends  over a finite set. 
\renewcommand\mono{q}
In addition to the spatial charge density $\mathsf q$ associated to the body charges, we will consider a time-dependent {(in our proof later we shall assume that $\mathsf q_{\rm ext}$ is time independent)} external charge density $\mathsf q_{\rm ext}(t,\mathsf x)$. We shall assume that the total density $\mathsf q+\mathsf q_{\rm ext}$ determines a  potential $\upphi(t,x)$ as the solution of the regularized Poisson equation
\begin{align}\label{gravity-large-strain}
-\operatorname{div}(\varepsilon(|\nabla\upphi|)\nabla\upphi)=\mathsf{\mono}
+\mathsf{\mono}_{\rm ext}\qquad\text{in }\R^d,
\end{align}
on the whole space, with vanishing conditions at infinity:
\begin{equation}\label{eq:57}
\lim_{|\mathsf x|\to+\infty}\upphi(t,x)=0\DELETEDELETE{,}{.}
\end{equation}
\DELETEDELETE{and with 
	$1_{S}$ denoting the characteristic function of a set $S$.} Here we assume that 
the effective permittivity is
\begin{equation}
\varepsilon(r)=\varepsilon_0(1+\epsilon_r|r|^{p-2})\qquad\forall r\ge 0,
\end{equation}
where $\varepsilon_0$ is the permittivity of vacuum,
i.e.\ $\varepsilon_0\doteq{\rm 8.854\times10^{-12}Fm^{-1}}$, 
and $\epsilon_r$ is a coeffient in a position like the relative (nonlinear) 
permittivity having the physical dimension $({\rm V}/{\rm m})^{2-p}$ 
with $p>d$ is a regularization exponent. 

The weak solution of \eqref{gravity-large-strain}--\eqref{eq:57} with $\mathsf q$ given by \eqref{eq:55} is the unique stationary point with respect to $\upphi$ of the electrostatic energy
\begin{align}\label{eq:65}
{\calE}_{\rm elec}(\bm\chi,\upphi)=&
%-\int_{\R^d}\frac {\varepsilon_0} 2 \big|\nabla\upphi\big|^2+\frac {\varepsilon_1} p \big|\nabla\upphi\big|^p\d {\mathsf x}+
\int_{\mathbb R^d} \Bigg(\sum_{x\in\stackrel{\leftarrow}{\bm\chi}(t,\mathsf x)}\frac{q(x)}{\det(\nablait\bm\chi(t,x))}\Bigg)\upphi
(\mathsf x)-\frac {\varepsilon_0} 2 \big|\nabla\upphi\big|^2-\frac {\varepsilon_1} p \big|\nabla\upphi\big|^p\,\d{\mathsf x},
\end{align}
where we have set $\varepsilon_1=\epsilon_r\varepsilon_0$.

\DELETEDELETE{Equation \eqref{eq:65} is not particularly convenient to us, because of the presence of the inverse function  $\stackrel{\leftarrow}{\bm\chi}$. However, an}{An} application of the change-of-variable formula \eqref{eq:54} with $f(x)=
\mono(x)\upphi(\bm\chi(t,x))/\det(\nablait\bm\chi(t,x))$ yields {(under the assumption that the deformation gradient is Lipschitz-continuous and that determinant of its  gradient is bounded from below by a positive constant, the function $f$ is integrable)}
\begin{align}\nonumber
\int_{\mathbb R^d}\Bigg(\sum_{x\in\stackrel{\leftarrow}{\bm\chi}(t,\mathsf x)}\frac{\mono(x)}{\det(\nablait\bm\chi(t,x))}\Bigg)\upphi(t,\mathsf x)\operatorname{\d\mathsf x}&=\int_{\mathbb R^d}\Bigg(\sum_{x\in\stackrel{\leftarrow}{\bm\chi}(\mathsf x)}
\frac{\mono(x)\upphi(t,\bm\chi(t,x))}{\det(\nablait\bm\chi(t,x))}\Bigg)\operatorname{\d\mathsf x}\\
&=\int_{\varOmega}\mono(x)\upphi(t,\bm\chi(t,x))\,\d x.
\end{align}
\DELETEDELETE{We note that, in view of the integrability of $q(x)$, the above choice yields an integrable function $f(x)$ because the solutions of \eqref{gravity-large-strain} are continuous and, as we shall see later, the solutions we find have the property that the determinant of the deformation gradient is bounded from below by a positive constant.} Using the last result and adding the mechanical and electrostatic energy, 
we obtain the total free energy:
\begin{align}%\nonumber
\calE(\bm\chi,\upphi)=&
\underbrace{\int_\varOmega\!\varphi(\nablait\bm\chi)
	\,\d x
	+\mathscr{H}(\nablait^2\bm\chi)
	%-\int_\varGamma\bm g\cdot\bm\chi\d S
}_{\displaystyle \calE_{\rm mech}(\bm\chi)}
%\\&
\ +\underbrace{\int_{\varOmega}\!\mono\upphi\circ\bm\chi\,\d x
	-\int_{\R^d}\!%\mathsf{\mono}_{\rm ext}\upphi
	\frac {\varepsilon_0}2 \big|\nabla\upphi\big|^2
	+\frac {\varepsilon_1} {p}\big|\nabla\upphi\big|^p\,\d\mathsf x
}_{\displaystyle \calE_{\rm elec}(\bm\chi,\upphi)}\,,
\label{static-gravity-functional+}
\end{align}
\color{black}
To make the comparison with the existing literature easier, we observe
that, starting from \eqref{static-gravity-functional+}, it is
customary to rewrite the electrostatic energy as an integral over the
reference configuration. In particular, defining the Lagrangean
potential $\phi=\upphi\circ\bm\chi$, the total energy would
take the form
written as
\begin{align}%\nonumber
\calE_{\rm Lagr}(\bm\chi,\phi)=\int_\varOmega\left[\psi(\nablait\bm\chi,\nablait\phi)-q\phi\right]\d x
+\mathscr{H}(\nablait^2\bm\chi),
\label{static-gravity-functional++}
\end{align}
with the Lagrangean energy density $\psi(\mathbf F,\mathbf e)=\varphi(\mathbf
F)+\varphi_{e}(\mathbf e)$  depending separately on the deformation
gradient $\mathbf F$ and on the Lagrangean electric field $\mathbf
e=-\nablait\phi$. As pointed out in the Introduction, this assumption
corresponds to adopting the ideal
dielectric model.
\color{black} % HERE ADDED SOME TEXT
Sometimes, $\calE$ is called an electrostatic Lagrangean because
of its saddle-point-like  structure\color{black}, cf.\ e.g.\
\cite[Sect.3.2]{ProWet09PEMF}.\color{black}
%\COMMENT{A REFERENCE?}
%\marginpar{\footnotesize The notion of saddle-point would be appropriate if $\mathscr E(\cdot,\upphi)$ was convex. However, this is not the case. Thus I wonder why we are speaking about a saddle point. --- I ADDED AT LEAST ``like''. IS IT BETTER?}  

We now define the kinetic energy as
\begin{align}\label{kinetic}
{\mathscr T}(\DT{\bm\chi})=\int_\varOmega\frac\varrho2|\DT{\bm\chi}|^2\,\d x,
\end{align}
where $\varrho=\varrho(x)$ with $\inf_{x\in\varOmega}\varrho(x)>0$ is 
the mass density. The (conventional) Lagrangean is then defined as 
\begin{align}\label{Lagrangean}
{\mathscr L}\big(\bm\chi,\DT{\bm\chi},\upphi\big):={\mathscr T}\big(\DT{\bm\chi}\big)-\calE(\bm\chi,\upphi).
\end{align}
Considering still an external electromechanical loading 
$\calF=(\calF_{\rm m}(t),\calF_{\rm e}(t))$ defined as a linear functional 
\begin{align}\label{load}
\langle\calF_{\rm m}(t),{\bm\chi}\rangle
=\int_\varOmega \bm f(t)\cdot\bm\chi\,\d x
+\int_\varGamma\bm g(t)\cdot\bm\chi\,\d S\quad\text{ and }\quad
\langle\calF_{\rm e}(t),\upphi)\rangle
=\int_{\R^d}\mathsf{\mono}_{\rm ext}\upphi\,\d\mathsf x,
\end{align}
the governing equation(s) can then be derived
by the {\it Hamilton variational principle} which 
\DELETEDELETE{says}{asserts} that, among all admissible motions on a fixed time interval $[0,T]$, 
the actual motion is such that the integral 
\begin{align}\label{hamilton}
\int_0^T{\mathscr L}\big(\bm\chi(t),\DT{\bm\chi}(t),\upphi(t)\big)
+\langle\calF(t),({\bm\chi},\upphi)(t)\rangle\,\d t
\ \
&\mbox{ is stationary,\ }
\end{align}
\mbox{i.e.~$(\bm\chi,\upphi)$ is its \emph{critical point}}.
This gives the system
\begin{subequations}\label{abstract-system}
	\begin{align}\label{abstract-system1}
	{\mathscr T}'\DDT{\bm\chi}+\partial_{\bm\chi}\calE(\bm\chi,\upphi)=\calF_{\rm m},
	\\\label{abstract-system2}
	\partial_{\upphi}\calE(\bm\chi,\upphi)=\calF_{\rm e},
	\end{align}\end{subequations}
where ${\mathscr T}'$ denotes the G\^ateaux derivative of the quadratic 
functional ${\mathscr T}$, i.e.\ a linear operator, and $\partial_{\bm\chi}\calE$ 
and $\partial_{\upphi}\calE$ are the respective partial G\^ateaux differentials.
The abstract ``algebraic'' part \eq{abstract-system2} forms a
{\it holonomic constraint}.

Taking into account the specific forms of the energies 
\eqref{static-gravity-functional+} 
and \eqref{kinetic}, the abstract system \eqref{abstract-system} results to\COMMENT{Shouldn't we use $div$ instead of ${\rm div}$ for the reference configuration?}
\begin{subequations}\label{system}\begin{align}\label{system1}
	&\varrho\DDT{\bm\chi}-{\rm div}\,{\stress}+\mono\nabla\upphi\circ\bm\chi=\bm f\ \ \ \
	\text{ with }\ {\stress}=\varphi'(\nablait\bm\chi)-{\rm div}\,\mathfrak{H}(\nablait^2\bm\chi),
	\\\label{system2}
	&\operatorname{div}((\varepsilon_0+\varepsilon_1|\nabla\upphi|^{p-2})\nabla\upphi)+\!\!\!\sum_{x\in\stackrel{\leftarrow}{\bm\chi}(\cdot,t)}\!
	\frac{\mono(x)}{\det(\nablait\bm\chi(t,x))}
	+\mathsf{\mono}_{\rm ext}(t,\cdot)=0
	\end{align}\end{subequations}
for a.a.\ $t\in I$, 
%\COMMENTGT{Here should be $-\kappa\Delta\upphi=%\chi_{\bm\chi(\varOmega)}^{}
%  \sum_{x\in\stackrel{\leftarrow}{\bm\chi}(\cdot,t)}\frac{\mono(x)}{\det(\nablait\bm\chi(x))}+\mathsf m_{\rm ext}$}
with $\nabla\upphi$ \DELETEDELETE{and $\Delta\upphi$} understood as space derivatives in the 
actual space configuration while other time/space derivatives are in the 
reference configuration, and the nonlocal hyperstress 
$\mathfrak{H}(\nablait^2\bm\chi)$ is given by {\eqref{eq:76} and \eqref{eq:75}}.
%\begin{align}
%\label{higher-nonlocal-stress}
%\big[\mathfrak{H}(\nablait^2\bm\chi)\big](x)
%=\int_\varOmega
%\mathfrak{K}(x,\tilde x)\Vdots\big(\nablait^2\bm\chi(x)-\nablait^2\bm\chi(\tilde x)\big)
%\,\d\tilde x\,.
%\end{align}
Note that the force in \eqref{system1} can be seen from 
\eqref{static-gravity-functional+} while the right-hand side of 
\eqref{system2} can be seen from the equivalent form \eqref{static-gravity-functional}. 
This system is augmented with the following boundary conditions 
\begin{subequations}\label{BC}
	\begin{align}\label{BC1}
	&\stress\bm n
	-\divS\mathfrak{H}(\nablait^2\bm\chi)=\bm g\ \ \text{ and }\ \ \mathfrak{H}(\nablait^2\bm\chi){:}(\bm{n}\otimes\bm{n})=0\qquad\text{on}\qquad\Sigma\\
	&\!\!\!\lim_{|\mathsf x|\to\infty}\upphi(t,\mathsf x)=0\quad\qquad \text{for all }t\in(0,T).
	\end{align}
\end{subequations}
and by the initial conditions
\begin{equation}\label{IC}
\bm\chi(\cdot,0)=\bm\chi_0,\qquad \DT{\bm\chi}(\cdot,0)=\bm v_0\qquad\text{in}\qquad\varOmega.
\end{equation}
%\COMMENT{+ B.C}

As there is no dissipation energy, the system \eqref{abstract-system} is (at 
least formally) conservative, i.e.~it conserves the mechanical energy 
${\mathscr T}\big(\DT{\bm\chi}\big)+\calE(\bm\chi,\upphi)$
\DELETEDELETE{conserved} during the evolution provided that {the forcing term vanishes}:
$\calF=0$. As there no kinetic energy associated with  $\upphi$-variable,  
%\COMMENT{Evolution is not only through inertia. IS IT NOW BETTER WITH VISCOSITY?? }
\eq{abstract-system} has the structure of an 
\emph{abstract differential-algebraic equation} (DAE) with $\upphi$ the 
``fast'' variable and $\bm\chi$ the ``slow'' 
variable. We conclude this section by showing that this system can be reduced to a single evolutionary partial differential equation for the deformation $\bm\chi$.

To begin with, we observe that $\calE(\bm\chi,\cdot)$ is strictly concave. This feature makes it possible for us to solve \eqref{abstract-system2} with respect to $\upphi$. The most convenient way to achieve this goal by means of the Legendre transform
% \DELETEDELETE{of $-\mathscr E$ with respect to the variable $\phi$}.
More specifically, we introduce the quantity
\begin{equation}
\mathscr E_{\rm tot}(\bm\chi,\upphi)=\mathscr E(\bm\chi,\upphi)-\langle\mathscr F,(\bm\chi,\upphi)\rangle,
\end{equation}
and then we let
\begin{equation}
[-\mathscr E_{\rm tot}(\bm\chi,\cdot)]^*(\bm\chi,\xi)=\sup_{\phi\in\R}\left(\langle\xi,\phi\rangle+\mathscr E_{\rm tot}(\bm\chi,\phi)\right),
\end{equation}
where, in the second formula, $\langle\cdot,\cdot\rangle$ denotes the duality between $W^{1,p}(\R^d)$ and $W^{-1,p}(\R^d)$. Given $\xi\in W^{-1,p}(\R^d)$, let us define $\upphi_\xi$ to be the solution of the following problem:
\begin{equation}
\varepsilon_0\Delta\upphi_\xi+\varepsilon_1\operatorname{div}(|\nabla\upphi_\xi|^{p-2}\nabla\upphi_\xi)
+\!\!\!\sum_{x\in\stackrel{\leftarrow}{\bm\chi}(\cdot,t)}\!
\frac{\mono(x)}{\det(\nablait\bm\chi(t,x))}
+\mathsf{\mono}_{\rm ext}=-\xi,
\end{equation}
with boundary conditions vanishing at infinity. Then we have 
\begin{equation}
[-\mathscr E_{\rm tot}(\bm\chi,\cdot)]^*(\bm\chi,\xi)=\int_{\R^d}\frac{\varepsilon_0}{2}|\nabla\upphi_\xi|^2+\frac{\varepsilon_1}{p'}|\nabla\upphi_\xi|^p \d \mathsf x,
\end{equation}
a representation formula which shows the convex character of the dual of $\mathscr E_{\rm tot}(\bm\chi,\cdot)$. It is a now standard result from Convex Analysis that
$$
\upphi=-\partial_\xi\big[-\calE_{\rm tot}(\bm\chi,0)\big]^*,
$$
so that, by substitution into \eqref{abstract-system1}
\begin{align}\label{underlying-ODE+}
{\mathscr T}'\DDT{\bm\chi}+
\partial_{\bm\chi}\calE\Big(\bm\chi\,,\,-\partial_\xi\big[-\calE_{\rm tot}(\bm\chi,0)\big]^*\Big)=\calF_{\rm m}.
\end{align}

\section{Analysis of the model by the Galerkin approximation}\label{sec-mono-anal}
%        ~~~~~~~~~~~~~~~~~~~~~~~~~~~~~~~

An important attribute of the model is that $\det(\nablait\bm\chi)$ occurs in 
the denominators in \eqref{system2}. \COMMENT{This is not exactly true, because in the weak formulation below we eliminate the determinant from the denominator!!} This requires to have control over $1/\det(\nablait\bm\chi)$, which can be 
ensured by having the free energy $\varphi=\varphi(\bm F)$ blowing up to 
$+\infty$ sufficiently fast if $\det\bm F\to0+$. More specifically, together 
with frame indifference, we assume altogether:
\begin{subequations}\label{ass-1}
	\begin{align}
	&\varphi:\R^{d\times d}\to[0,+\infty]\ \ \text{ continuously differentiable on }{\rm SL}^+(d),\label{cont}
	\\
	\label{frame}
	&\forall\bm Q\in{\rm SO}(d):\quad\varphi(\bm Q\bm F)=\varphi(\bm F),
	\\&\label{ass-HK}
	%{\rm det}\bm F>0\ \ \ \Rightarrow\ \ \ 
	\varphi(\bm F)\begin{cases}\ge\epsilon/({\rm det}\bm F)^p\!\!&\text{if 
	}\ {\rm det}\bm F>0,\\
	=+\infty &\text{if 
	}\ {\rm det}\bm F\le0,\end{cases}\ \text{ for some}
	\ \ 
	p>\frac{2d}{d{-}2{-}2\gamma}
	%p>\frac{2d}{d{-}2{-}2\gamma}
	%\frac6{2\gamma{-}1}
	,\ \ \gamma>\frac d 2-1.
	%d/2{-}1,
	%\\&{\rm det}\bm F\le0\ \ \ \Rightarrow\ \ \ \psi(\bm F,\bm m,\zeta,\theta)=\infty,
	%\\&\label{ass-HK+}{\rm det}\bm F\le0\ \ \ \Rightarrow\ \ \ \varphi(\bm F)=+\infty.
	\end{align}\end{subequations}
where $\gamma$ refers to \eqref{ass-kernel-K+}. Concerning the regularizing 
kernel $\mathfrak K$, 
%the integrand 
we assume
\begin{align}\label{ass-kernel-K+}
\!\exists\eps>0\ \forall x,\tilde x\!\in\!\varOmega,\ G\!\in\!\R^{d\times d}:
\quad
\bigg(\frac{\eps|G|^2}{|x{-}\tilde x|^{d+2\gamma}\!\!\!}\ \ -\frac1\eps\bigg)^+\!\!\!
\le G{:}
%\Vdots
\mathfrak{K}(x,\tilde x)
%\Vdots 
{:}G
\le\frac{|G|^2}{\eps|x{-}\tilde x|^{d+2\gamma}\!\!\!}\ .
\end{align}
Furthermore, we shall make the following assumptions concerning the initial data and the body loading and the surface loading
\begin{subequations}\label{eq:56}\begin{align}\label{eq:56-IC}
	&\bm\chi_0\in H^{2+\gamma}(\Omega;\mathbb R^d)\quad\text{with}\quad
	%\mathscr E_{\rm tot}(\bm\chi_0)<\infty
	\varphi(\nablait\bm\chi_0)\in L^1(\varOmega), \qquad 
	\bm v_0\in L^2(\Omega;\mathbb R^d),\qquad 
	\\&\label{eq:56-f-g}
	\bm f\in L^\infty(0,T;L^2(\varOmega;\R^d)),\qquad \bm g\in W^{1,1}(0,T;L^1(\varGamma;\mathbb R^d)).
	\end{align}\end{subequations}

The derivation of the weak formulation of \eqref{system1} is standard, but requires some care. We take the scalar product of both sides with a test velocity $\bm\zeta$ such that $\bm\zeta(T)=\bm 0$ and $\DT{\bm\zeta}(T)=\bm 0$, and we integrate over the domain $\varOmega$ and on the time interval $I$. Then we integrate by parts twice both on the domain $\varOmega$ and on the time interval $I$.  The result is
\begin{align}\nonumber
&\int_Q\varrho{\bm\chi}{\cdot}\DDT{\bm\zeta}
+\varphi'(\nablait\bm\chi):\nablait\bm\zeta+\mathfrak H(\nablait^2\bm\chi)\vdots\nablait^2\bm\zeta
\,\d x\d t
%+\int_Q(\nablait^2\bm\chi(x')-\nablait^2\bm\chi(x))\vdots\mathscr K(x-x')\vdots\nablait^2\bm\zeta(x)\,\d x\d x
\\[-.2em]&\ 
=-\int_Q q(\nabla\upphi\circ\bm\chi)\cdot\bm\zeta\,\d x\d t
+\int_{\varSigma}\bm g\cdot\bm\zeta\,\d S\d t
+\int_{\varOmega}\varrho\bm v_0{\cdot}\bm \zeta(0)-\varrho\bm\chi_0{\cdot}\DT{\bm\zeta}(0)\,\d x.
\end{align}
In our proof of existence of weak solutions, a passage to the limit in the integral term involving $\nabla\upphi$ on the right--hand side would be problematic. Indeed, the term $\nabla\upphi$ is composed with the deformation $\bm\chi$, and in order to pass to the limit in our approximation procedure even strong convergence of $\nabla\upphi$ in a $L^p$ space would not suffice. Indeed, we would need convergence of $\nabla\upphi$ in the space of continuous functions, which however cannot be expected. For this reason, we have to rewrite the aforementioned term through a few manipulations. As a start, we observe that $(\nabla\upphi)\circ\bm\chi=\nablait\bm\chi^{-\top}\nablait(\upphi\circ\bm\chi)$. Thus, using integration by parts we obtain
\begin{align}\nonumber
&\int_Q q(\nabla\upphi\circ\bm\chi)\cdot\bm\zeta\,\d x\d t
=\int_Q q\bm\zeta\cdot\nablait\bm\chi^{-\top}\nablait(\upphi\circ\bm\chi)
=\int_Q \nablait(\upphi\circ\bm\chi)\cdot\nablait\bm\chi^{-1}(q\bm\zeta)=
\\[-.3em]
&\qquad=\int_\Sigma q(\upphi\circ\bm\chi)\nablait\bm\chi^{-\top}\bm n\cdot\bm \zeta
-\int_Q \upphi\circ\bm\chi \operatorname{div}(q\nablait\bm\chi^{-1}\bm\zeta)\,\d x\d t\nonumber\\\nonumber
&\qquad=
\int_\Sigma q(\upphi\circ\bm\chi)\nablait\bm\chi^{-\top}\bm n\cdot\bm \zeta
-\int_Q (\upphi\circ\bm\chi)\nablait\bm\chi^{-\top}\cdot\nabla (q\bm\zeta)\,\d x\d t
\\[-.4em]&\qquad\qquad\qquad\qquad\qquad\qquad\qquad\qquad
-\int_Q q(\upphi\circ\bm\chi)\bm\zeta\cdot\operatorname{div}(\nablait\bm\chi^{-\top})\,\d x\d t.
\end{align}
We can now write our notion of weak solution.

\begin{definition}[Weak solution to the initial-boundary-value
	problem \eq{system}--\eq{BC}--\eq{IC}]\label{def}
	%\COMMENT{TO BE STILL CHANGED}
	A pair $(\bm\chi,\upphi)\in C_{\rm w}(I;H^{2+\gamma}(\varOmega;\R^d))\times 
	L^\infty(I;W^{1,p}(\mathbb R^d))$ is a weak solution to \eq{system}, \eq{BC}, and 
	\eq{IC} if $\mathscr{H}(\nablait^2\bm\chi)\in L^\infty(I)$,
	if $\DT{\bm\chi}\in C_{\rm w}(I;L^2(\varOmega;\R^d))$,
	and if the following conditions hold:
	\medskip
	
	\noindent 1) for every $\bm \zeta\in C^\infty(\overline{Q};\mathbb R^d)$ 
	satisfying $\bm\zeta(T)=\DT{\bm\zeta}(T)=0$,
	\begin{subequations}\label{w-sln}
		\begin{align}\nonumber
		&\int_Q\varrho{\bm\chi}{\cdot}\DDT{\bm\zeta}
		+
		%\partial_{\bm F}
		\varphi'(\nablait\bm \chi)
		{:}\nablait\bm\zeta
		+
		%\bigg(\int_\varOmega(\nablait^2\bm\chi(x')-\nablait^2\bm\chi(x))\vdots\mathscr K(x-x')\,\d x'\bigg)
		\mathfrak{H}(\nablait^2\bm\chi)
		\Vdots\nablait^2\bm\zeta(x)
		\,\d x\d t
		%+\int_Q(\nablait^2\bm\chi(x')-\nablait^2\bm\chi(x))\vdots\mathscr K(x-x')\vdots\nablait^2\bm\zeta(x)\,\d x\d x
		=\int_Q\!(\upphi\circ\bm\chi)\nablait\bm\chi^{-\top}\!:\nabla(q\bm\zeta)
		\\[-.4em]&\ \ 
		+
		\bigg(q(\phi\circ\bm\chi)\operatorname{div}(\nablait\bm\chi^{-\top})+\bm f\bigg)
		{\cdot}\bm\zeta\,\d x\d t
		+\int_{\varSigma}(q(\phi\circ\bm\chi) \nablait\bm\chi^{-\top}\bm n+\bm g)\cdot\bm\zeta\,\d S\d t\nonumber
		% +\int_{\varOmega}\varrho\bm v_0{\cdot}\bm \zeta(0)-\varrho\bm\chi_0{\cdot}\DT{\bm\zeta}(0)\,\d x;
		\\[-.4em]
		& \hspace{18em}\
		+\int_\varOmega \varrho\bm v_0\cdot\bm\zeta(0)-\varrho\bm\chi_0\cdot\DT{\bm\zeta}(0)\d x,
		\label{w-sln1}\end{align}
		with the hyperstress
		$$[\mathfrak{H}(\nablait^2\bm\chi)](t,x)
		=\sum_{i=1}^d\bm e_i\otimes\int_\varOmega\mathfrak K_i(x,x'){:}
		(\nablait^2\chi_i(t,x'){-}\nablait^2\chi_i(t,x))
		%\vdots
		\,\d x'$$ a.e.\ on $Q$, \emph{cf.} \eqref{eq:76} and \eqref{eq:75}.
		
		\medskip
		
		\noindent 2) For every $\zeta\in C_0^\infty(I\times\mathbb R^d)$,
		\begin{align}
		&\int_{I\times{\mathbb R^d}}(\varepsilon_0+\varepsilon_1|\nabla\upphi|^{p-2})\nabla\upphi\cdot\nabla\zeta
		-\mathsf{\mono}_{\rm ext}\zeta\,\d \mathsf x\d t=
		\int_{Q}\!q
		(\zeta{\circ}\bm\chi)\,\d x \d t.
		%+ \int_{I\times\mathbb R^d}\mathsf{\mono}_{\rm ext}\zeta\,\d \mathsf x\d t.
		\label{w-sln3}
		\end{align}\end{subequations}
\end{definition}
If $\varphi$
%was
were 
semiconvex, that is, if $\varphi''$ was bounded from below, 
we could use a standard technique (see for instance \cite[Remark 9.5.4]{KruRou18MMCM}), based on 
the a time discretization to provide a constructive proof of existence of 
solutions. However, in the present case semiconvexity is incompatible with 
the frame-indifference requirement \eqref{frame}. Instead, we will resort 
to the Galerkin method to construct approximate solutions. 
{It is important that the singularity of the free energy for 
	$\det\bm F\to0+$ is eliminated together with the {spurious} 
	Lavrentiev phenomenon (cf.\ \cite{FoHrMi03LGPN}) by the used 
	nonsimple-material concept in cooperation with the Healey-Kr\"omer theorem 
	\cite{HeaKro09IWSS}.}
{It is important that this can be done already on the level of Galerkin 
	approximation, so that no other regularization of the singularity 
	of $\varphi$ at $\det\bm F\to0+$ is not needed,} cf.\ also \cite{RouTom2018}.
%\marginpar{\small\bfIt would be important to explain at what point of the proof this is used.}

To this goal, we take \REPLACE{increasing}{nested} (with respect to 
``$\subset$'') sequences (indexed by $k\in\N$) of some finite-dimensional 
subspaces $X_k$ of $H^{2+\gamma}(\varOmega;\R^d)$  whose union is dense, i.e.:
\begin{equation}\label{nested-X}
X_k\subset X_{k+1},\qquad \overline{\bigcup_k X_k}=H^{2+\gamma}(\varOmega;\R^d).
\end{equation}
Without loss of generality, we may assume that $X_0$ is spanned by the initial configuration $\bm\chi_0\in H^{2+\gamma}(\varOmega;\R^d)$ and
by the initial velocity $\bm v_0\in L^2(\varOmega;\R^d)$. Then for each $k\in\mathbb N$ we solve the following approximation of 
%\eqref{w-sln1}--\eqref{w-sln3}
\eq{w-sln}, which we write in an abstract form, in the same fashion of \eqref{abstract-system}:
\begin{subequations}\label{abstract-system-discrete}
	\begin{align}\label{abstract-system-discrete1}
	&\langle{\mathscr T}'\DDT{\bm\chi}_k+\partial_{\bm\chi}\calE(\bm\chi_k,\upphi_k),\bm v\rangle=\langle\mathscr F_{\rm m},\bm v\rangle\qquad\forall \bm v\in X_k,
	\\
	\label{abstract-system-disc2}
	&\partial_{\upphi}\calE(\bm\chi_k,\upphi_k)-\mathscr F_{\rm e}=0.
	\end{align}\end{subequations}
It should be pointed out that the Galerkin approximation applies 
only to \eq{abstract-system-discrete1} while \eq{abstract-system-disc2} is kept
continuous.
The existence of solutions for the above system can be obtained by following a path similar to that  leading to \eqref{underlying-ODE+}, which leads to
\begin{align}\label{underlying-ODE++}
\left\langle{\mathscr T}'\DDT{\bm\chi}_k+
\partial_{\bm\chi}\calE\Big(\bm\chi_k\,,\,-\partial_\xi\big[-\calE_{\rm tot}(\bm\chi_k,0)\big]^*\Big),\bm v\right\rangle=0\qquad\forall \bm v\in X_k,
\end{align}
which is equivalent to a system of ordinary differential equations whose solution for small times can be proved by standard arguments. When the solution of \eqref{underlying-ODE++} is obtained, we define the approximants $\upphi_k$ as
% \COMMENT{NOW ALSO $\calF_{\rm e}$ TO BE ADDED}
\begin{equation}\label{eq:66}
\upphi_k=-\partial_\xi[-\mathscr E_{\rm tot}(\bm\chi_k,0)].
\end{equation}
The following result by T.J.\,Healey and S.\,Kr\"omer \cite{HeaKro09IWSS} is of 
an essential importance: Formulating it in an arbitrary dimension $d\in\N$, 
let $p>d$, $q\ge pd/(p{-}d)$, and $K\in\R$, then there is 
$\eps=\eps(p,q,K)>0$ such that
%Let $C>0$ be such that 
\begin{align}%\nonumber
%S_c:=\bigg\{
\left.\begin{array}{l}\bm\chi\in W^{2,p}(\varOmega;\R^d)
\\
\det(\nablait\bm\chi)>0\ 
\mathrm{a.e. in}\  \varOmega%\ \text{ and }\ 
\\\|\bm\chi\|_{W^{2,p}(\varOmega;\R^d)}+\!\int_\varOmega%\frac
1/{(\det\nablait\bm\chi(x))^{q}}\,\d x\le K
\end{array}\right\}
%\bigg\}
%\ne\emptyset\ 
%\\
\quad\Rightarrow
%\end{align*}with $c$ given so that $S_c\ne\emptyset$. Then:\begin{align}
\label{Healey-Kromer-bound}
%\exists\,\eps=\eps(p,q,c)>0\ \ \forall\,\bm\chi\in S_c:
\quad
\det(\nablait\bm\chi)\ge\eps\ \ \text{ on }\ \bar\varOmega.
\end{align}
Here we will use it in a modification involving the quadratic form 
$\mathscr{H}$ as in \eq{8-nonlocal-f-hyper} with \eq{ass-kernel-K+} 
satisfied for some $\gamma>d/2-1$ and with a potential 
$\varphi:\R^{d\times d}\to\R$ satisfying $\varphi(\bm F)\ge1/(\det\bm F)^q$ for 
$q>2d/
%(d{-}2{-}2\gamma)
(2\gamma{+}2{-}d)$, cf.\ \cite[Sect.\,2.5]{KruRou18MMCM}:
Considering the functional 
$\Phi(\bm\chi):=\int_\varOmega\varphi(\nablait\bm\chi)\,\d x+\mathscr{H}(\nablait^2\bm\chi)$, for any $K$, there is $\eps>0$ such that
\begin{align}
\Phi(\bm\chi)\le K\qquad\Rightarrow\qquad \det(\nablait\bm\chi)\ge\eps\quad\text{ on $\bar\varOmega$.}
\label{det-positive-on-lev}\end{align}
%\DELETEDELETE{Here we will use it for $d=3$.}

\begin{proposition}[Weak solutions to \eqref{system}]\label{prop-1}
	Assume that $q\in W^{1,1}(\Omega)$, {that 
		$\mathsf q_{\rm ext}\in L^1(\mathbb R^d)$ is time independent,} and that $p>d$.  
	Then the Galerkin approximation $(\bm\chi_k,\upphi_k)$ exists on the whole 
	time interval $I$ and satisfies the a-priori estimates
	\begin{subequations}\label{est-1}
		\begin{align}
		&\|\bm\chi_k\|_{W^{1,\infty}(I;L^2(\varOmega;\R^d))\,\cap\,L^\infty(I;H^{2+\gamma}(\varOmega;\R^d))}\le C
		\ \ \text{ with }\ \ \Big\|\frac1{\det(\nablait\bm\chi_k)}\Big\|_{L^\infty(Q)}^{}\le C,\label{eq:58}
		\\&\|\upphi_k\|_{L^\infty(I;W^{1,p}(\R^d))}^{}\le C.\label{eq:60}
		\end{align}
	\end{subequations}
	Moreover, there is a subsequence of $\{(\bm\chi_k,\upphi_k)\}_{k\in\N}$ 
	converging weakly* in the topologies indicated in \eqref{est-1} and 
	the limit of any such subsequence is the weak solution to 
	\eqref{system} with the initial conditions $\bm\chi(0)=\bm\chi_0$
	and $\DT{\bm\chi}=\bm v_0$.
\end{proposition}

\begin{proof}
	By the Healey-Kr\"omer theorem \cite{HeaKro09IWSS}
	in its modification as \eqref{det-positive-on-lev}, for some positive 
	$\varepsilon\le\min_{x\in\bar{\varOmega}}\det\nablait\bm\chi_0(x)$, we have
	\begin{align}\label{pos-of-det}
	\forall (t,x)\in\bar Q:\qquad\det\nablait\bm\chi_k(t,x)\ge\varepsilon.
	\end{align} 
	%Actually, here we use a modification of the original result 
	%\cite{HeaKro09IWSS}, relying on the embedding of 
	%$H^{2+\gamma}(\varOmega)\subset W^{2,.........}(\varOmega)$,
	%cf.~\cite[Sect.\,2.5]{KruRou18MMCM}.\COMMENT{LIKELY ALREADY WILL BE ABOVE}
	For \eqref{pos-of-det}, we use
	the successive-continuation argument on the Galerkin level, and
	thus $\nablait\bm\chi_k$ is valued in the definition domain $\varphi$ and 
	the singularity of $\varphi$ is not seen. In particular, the 
	Lavrentiev phenomenon (which may occur when $\varphi$ would not 
	have enough fast growth to $+\infty$ if $\det\bm F\to0+$)
	is excluded. 
	
	Having a local solution of the Galerkin approximation, we test (\ref{abstract-system-discrete1}) and \eqref{abstract-system-disc2} by $\DT{\bm\chi}_k$ and $\DT{\upphi}_k$, respectively, we add the resulting equations and we use the chain 
	rule to obtain
	\begin{align}\nonumber
	&\int_\varOmega \frac\varrho 2 |\DT{\bm\chi}_k(\tau)|^2\d x
	+\mathscr E_{\rm mech}(\bm\chi_k(\tau))
	+\mathscr E_{\rm elec}(\bm\chi_k(\tau),\upphi_k(\tau))
	=\int_\varOmega \frac\varrho 2 |\DT{\bm v}_0|^2\d x
	\\
	&\quad+\mathscr E(\bm\chi_0,\upphi_0)
	+\int_0^\tau\!\!\int_\varOmega\bm f\cdot\DT{\bm\chi}_k\d x\d t\nonumber
	+\int_0^\tau\!\!\int_\varGamma{\bm g}\cdot\DT{\bm\chi}_k\d S\d t
	+\int_0^\tau\!\!\int_{\R^d}{\mathsf q}_{\rm ext}\DT\upphi_k\d\mathsf x \d t 
	\end{align}
	for all $\tau$ in some interval $(0,T_k)$ with $T_k\le T$. \COMMENT{$\tau$ BETTER $t$???} Then, we integrate by parts with respect to time the terms containing time derivatives of $\bm\chi_k$ and $\upphi_k$ that cannot be controlled by the energetic terms, namely, the terms involving $\bm g$ and $\mathsf q_{\rm ext}$, and we obtain 
	\begin{align}\label{eq:61}
	&\int_\varOmega \frac\varrho 2 |\DT{\bm\chi}_k(\tau)|^2\d x+\mathscr E_{\rm mech}(\bm\chi_k(\tau))-\int_\varSigma\bm g(\tau)\cdot\bm\chi_k(\tau)\, \d S
	+\mathscr E_{\rm elec}(\bm\chi_k(\tau),\upphi_k(\tau))\nonumber
	\\&\quad-\int_{\R^d}\mathsf q_{\rm ext}(\tau)\upphi_k(\tau)\d \mathsf x=\int_\varOmega \frac\varrho 2 |\DT{\bm v}_0|^2\d x+\mathscr E(\bm\chi_0,\upphi_0)
	+\int_0^\tau\!\!\int_\varOmega\bm f\cdot\DT{\bm\chi}_k\d x\d t\nonumber
	%-\int_0^\tau\!\!\int_{\R^d}\DT{\mathsf q}_{\rm ext}\upphi_k\d\mathsf x \d t\nonumber
	\\
	&\qquad -\int_0^\tau\!\!\int_\varGamma\DT{\bm g}\cdot{\bm\chi}_k\d S\d t
	-\int_\varSigma\bm g(0)\cdot\bm\chi_k(0)\, \d S
	-\int_{\R^d}\mathsf q_{\rm ext}(0)\upphi_k(0)\d \mathsf x.
	\end{align}
	We observe that, for every $\delta>0$,
	\begin{equation*}
	\begin{aligned}
	&\!\!\!\!\!\mathscr E_{\rm mech}(\bm\chi_k(\tau))-\int_\varGamma \bm g(\tau)\cdot\bm\chi_k(\tau)
	\\[-.4em]&
	\ge C_1 \|\nablait^2\bm\chi_k(\tau)\|_{L^2(\varOmega;\R^{d\times d\times d})}-\frac 1{2\delta} \|\bm g(\tau)\|_{L^1(\varGamma;\R^d)}^2-\frac \delta 2 \|\bm\chi_k(\tau)\|_{L^\infty(\varOmega;\R^d)}^2\\
	&\ge  C_1 \|\bm\chi_k(\tau)\|^2_{H^{2+\gamma}(\varOmega;\R^d)}-C_2\|\bm\chi_k(\tau)\|^2_{L^2(\varOmega;\R^d)} 
	\\&\qquad\qquad\qquad\qquad\qquad\qquad\qquad
	-\frac 1{2\delta} \|\bm g(\tau)\|_{L^1(\varGamma;\R^d)}^2-\frac \delta 2 \|\bm\chi_k(\tau)\|_{L^\infty(\varOmega;R^d)}^2.
	\end{aligned}
	\end{equation*}
	From the above equation, by using the embedding $H^{2+\gamma}(\varOmega)\subset C(\overline\varOmega)$, and by taking $\delta$ sufficiently small, we obtain
	\begin{equation}\label{eq:69}
	\mathscr E_{\rm mech}(\bm\chi_k(\tau))-\int_\varGamma \bm g(\tau)\cdot\bm\chi_k(\tau)\,\d S\ge C_1 \|\bm\chi_k\|^2_{H^{2+\gamma}(\varOmega;\R^d)}-C_2\Big(\|\bm\chi_k\|^2_{L^2(\varOmega;\R^d)}+1\Big).
	\end{equation}
	
	Next, we notice that if $\upphi_k(t)$ solves the nonlinear electrostatic equation \eqref{abstract-system-disc2} at time $t$, then a test by $\upphi_k(t)$ yields
	\begin{equation}
	\int_{\varOmega} q\upphi_k(t)\circ\bm\chi_k(t)\d x
	+\int_{\R^d}\mathsf q_{\rm ext}(t)\upphi_k(t) \d \mathsf x
	=\int_{\R^d}\varepsilon_0|\nabla\upphi_k(t)|^2
	+\varepsilon_1|\nabla\upphi_k(t)|^p \d \mathsf x.
	\end{equation}
	This implies that
	\begin{equation}\label{eq:59}
	\mathscr E_{\rm elec}(\bm\chi_k(t),\upphi_k(t))-\int_{\R^d}\mathsf q_{\rm ext}(\tau)\upphi_k(\tau)\d \mathsf x
	=\int_{\R^d}\frac {\varepsilon_0}2 |\nabla\upphi_k(t)|^2
	+\frac {\varepsilon_1}{p'}|\nabla\upphi_k(t)|^p \d \mathsf x.
	\end{equation}
	Furthermore, we have
	\begin{align}\label{eq:70}
	& \int_0^\tau\int_\varOmega\bm f\cdot\DT{\bm\chi}_k\,\d x \d t\le \|\bm f\|_{L^\infty(I;L^2(\varOmega;\R^d))}\int_0^\tau \|\DT{\bm\chi}_k\|_{L^2(\varOmega;\R^d)}\, \d t\nonumber\\&\le \|\bm f\|_{L^\infty(I;L^2(\varOmega;\R^d))}+\frac \tau 2+\frac 12 \int_0^\tau\|\DT{\bm\chi}_k\|^2_{L^2(\varOmega;\R^d)}\,\d t
	\le C\bigg(1+\int_0^\tau \|\DT{\bm\chi}_k\|^2_{L^2(\varOmega;\R^d)}\bigg).
	\end{align}
	Likewise, we have
	%  \begin{align}\label{eq:68}
	%    -\int_0^\tau\!\!\int_{\R^d}\DT{\mathsf q}_{\rm ext}\upphi_k\d\mathsf x \d t&\le \int_0^\tau\|\DT{\mathsf q}_{\rm ext}\|_{L^1(\R^d)}\|\upphi_k\|_{L^\infty(\R^d)}\, d t
	%\nonumber                                                                       \\&
	%\le \|\DT{\mathsf q}_{\rm ext}\|_{L^\infty(I;L^1(\R^d))}+\frac \tau {p'}+\frac 1 p \int_0^\tau \|\upphi_k\|^p_{L^\infty(\R^d)}\, \d t
	%      \le C\bigg(1+\int_0^\tau \|\upphi_k\|^p_{W^{1,p}(\R^d)}\, \d t\bigg),
	%  \end{align}
	%  and
	\begin{align}\label{eq:71}
	& -\int_0^\tau\!\!\int_\varGamma\DT{\bm g}\cdot{\bm\chi}_k\d S\d t\le \int_0^\tau\|\DT{\bm g}\|_{L^1(\varGamma;\R^d)}\|{\bm\chi}_k\|_{L^\infty(\varGamma;\R^d)}\d t \nonumber\\
	&\le \frac 12 \int_0^\tau\|\DT{\bm g}\|_{L^1(\varGamma;\R^d)}^2 \,\d t +\frac 12 \int_0^\tau\|{\bm\chi_k}\|^2_{L^\infty(\varGamma;\R^d)} \,\d t
	\le C\bigg(1+\int_0^\tau \|\bm\chi_k\|^2_{H^{2+\gamma}(\varOmega;\R^d))}\bigg).
	\end{align}
	By combining the estimates \eqref{eq:59}, \eqref{eq:70}, and \eqref{eq:71} we obtain
	%\COMMENT{$\|\bm\chi_k(t)\|^2_{H^\gamma(\varOmega;\R^{d\times d\times d})}$ TO BE STILL CORRECTED!!!!}
	\begin{align*}
	\|\DT{\bm\chi}_k(\tau)\|^2_{L^2(\varOmega;\R^d)}&+\|\bm\chi_k(\tau)\|^2_{H^{2+\gamma}(\varOmega;\R^{d})}+\|\upphi_k(\tau)\|^p_{W^{1,p}(\R^d)}\\
	&\le C\bigg(1+\|\bm\chi_k(\tau)\|^2_{L^2(\varOmega;\R^d)}+\int_0^\tau\|\DT{\bm\chi}_k(t)\|_{L^2(\varOmega;\R^d)}
	%       +\|\bm\chi_k(t)\|^2_{H^\gamma(\varOmega;\R^{d})}+ \|\upphi_k(t)\|^p_{W^{1,p}(\R^d)}\d t
	\bigg).
	\end{align*}
	Finally, we observe that
	\begin{align}\nonumber
	\|\bm\chi_k(\tau)\|^2_{L^2(\varOmega;\R^d)}&=\int_\varOmega \Big|\bm\chi_0+\int_0^\tau \DT{\bm\chi}_k(t)\dt\Big|^2\,\d x
	\\&
	\le 2 \int_\varOmega |\bm\chi_0|^2 \, d x+2\int_\varOmega\Big|\int_0^\tau \DT{\bm\chi}_k(t)\dt\Big|^2\,\d x\nonumber\\
	&
	\le 2 \int_\varOmega |\bm\chi_0|^2 \, d x+2\tau\int_\varOmega\int_0^\tau \big|\DT{\bm\chi}_k(t)\big|^2\dt\,\d x.
	\end{align}
	Thus, we can estimate
	\begin{align}\nonumber
	\|\DT{\bm\chi}_k(\tau)\|^2_{L^2(\varOmega;\R^d)}&+\|\bm\chi_k(\tau)\|^2_{H^{2+\gamma}(\varOmega;\R^{d})}+\|\upphi_k(\tau)\|^p_{W^{1,p}(\R^d)}\\
	&\le C\bigg(1+\int_0^\tau\|\DT{\bm\chi}_k(t)\|_{L^2(\varOmega;\R^d)}+\|\bm\chi_k(t)\|^2_{H^{2+\gamma}(\varOmega;\R^{d})}
	% + \|\upphi_k(t)\|^p_{W^{1,p}(\R^d)}\d t
	\bigg).\label{end-of-est}
	\end{align}
	At this point, by the application of Gronwall's inequality and by the standard continuation argument we can deduce the existence of approximate solutions on the whole time interval $I$, along with the bounds \eqref{eq:58} and \eqref{eq:60}.
	%\COMMENT{HERE \eqref{system-est} + ALL ARGUMENTATION AROUND (Gronwall etc)}
	%\COMMENT{PROBABLY THE ASSUMPTION ON THE TIME DIFFERENTIABILITY OF $\bm g$ IS SUBOPTIMAL, AND IN FACT YOUR DERIVATION DOES NOT REQUIRE TIME DIFFERENTIABILITY OF $\bm g$. HOWEVER, I DID NOT WANT TO MAKE TOO MANY CHANGES. -- NOT TRUE!!!! 
	%SEE THE SENTENCE: ``The term $\int_\Gamma g\cdot\DT{\bm\chi}_k\,\d S$ in 
	%\eq{system-est} is to
	%be treated by the by-part integration when  \eq{system-est} is integrated 
	%in time over the interval $[0,t]$.'' BELOW \eq{system-est}. PERHAPS IT SHOULD
	%BE WRITTEN MORE IN DETAIL, BUT IT IS RELATIVELY CLEAR HOW TO MAKE IT.}

	By Banach's selection principle, there exist $\bm\chi\in L^\infty(I;H^{2+\gamma}(\varOmega;\mathbb R^d))\cap W^{1,\infty}(I;L^2(\varOmega;\mathbb R^d))$ and $\upphi\in L^\infty(I;W^{1,p}(\R^d))$ such that, by possibly extracting a subsequence (which we do not relabel),
	\begin{subequations}\label{eq:62}
		\begin{align}
		&\bm\chi_k\to\bm\chi \qquad\text{weakly* in }L^\infty(I;H^{2+\gamma}(\varOmega;\mathbb R^d))\cap W^{1,\infty}(I;L^2(\varOmega;\mathbb R^d)),\label{eq:62a}\\
		&\upphi_k\to\upphi \qquad\text{weakly* in }L^\infty(I;W^{1,p}(\varOmega)).\label{eq:62b}
		\end{align}
	\end{subequations}
	\DELETEDELETE{The last step of the proof is the limit passage in the definition of weak 
		solution, see \eq{}.}
	
	%\eqref{w-sln1}--\eqref{w-sln3}. 
	By the compact embedding $H^{2+\gamma}(\varOmega)\subseteq C^1(\overline\varOmega)$ and by the continuous embedding $C^1(\overline\varOmega)\subset L^2(\overline\varOmega)$, the Aubin-Lions theorem and \eqref{eq:62a} imply that
	\begin{equation}\label{eq:63}
	\bm\chi_k\to \bm\chi\quad\text{strongly in }L^r(I;C^1(\overline\varOmega;\R^{d}))\qquad\forall 1\le r<\infty.
	\end{equation}
	This also implies that the restriction of $\nablait\bm\chi_k$ on $\varGamma$ converges strongly in $L^r(I;C(\varGamma;\R^{d\times d}))$. Consider the function $\bm F\mapsto\bm F^{-\top}=\operatorname{Cof}\bm F/\det\bm F$ on the set of all \REPLACE{$d$-by-$d$ }{$d{\times}d$-}matrices whose determinant is bounded from below by a positive constant. This function induces a continuous map from $C(\overline\varOmega;\R^{d\times d})$ to itself. Hence, from \eqref{eq:63} we have
	\begin{equation}
	\nablait\bm\chi_k^{-\top}\to\nablait\bm\chi^{-\top}\quad\text{strongly in }L^r(I;C(\overline\varOmega;\R^{d\times d}))\qquad\forall 1\le r<\infty.
	\end{equation}
	Next, it follows from the $W^{1,\infty}$-estimate in \eqref{eq:58} that 
	\begin{equation}
	\|\nablait^2\DT{\bm\chi}_k\|_{L^\infty(I;H^2(\varOmega;\R^{d\times d\times d})^*)}\le C.
	\end{equation}
	Also, we have that $H^{\gamma}(\varOmega)$ is compactly embedded in 
	$L^2(\varOmega)$. Thus, in view of the $L^\infty$ estimate in \eqref{eq:58} we 
	can again apply the Aubin-Lions theorem to $\nablait^2\bm\chi$ to deduce that
	\begin{equation}
	\nablait^2\bm\chi_k\to\nablait^2\bm\chi\qquad\text{strongly in }L^r(I;L^2(\varOmega;\R^{d\times d\times d}))\qquad\forall 1\le r<\infty.
	\end{equation}
	The strong convergence statement \eqref{eq:63} can be written as $\int_0^T \|\bm\chi_k(t)-\bm\chi(t)\|^r_{C^1(\overline\varOmega;\R^d)}\to 0$ as $k\to\infty$. This also implies that 
	\begin{equation}\label{eq:72}
	\bm\chi_k(t)\to\bm\chi(t)\qquad\text{strongly in }C^1(\overline\varOmega;\R^d)\text{ for a.a. }t\in(0,T).
	\end{equation}
	We also know from the uniform estimate \eqref{eq:60} that for each $k$ there exists a set $I_k$ such that $\text{meas}(I\setminus I_k)=0$ and such that  $\|\upphi_k(t)\|_{W^{1,p}(\R^d)}\le C$ for all $t\in I_k$. We let $J=\cap_k I_k$. Then $\text{meas}(I\setminus J)=0$, and 
	$\|\upphi_k(t)\|_{W^{1,p}(\R^d)}\le C$ for all $t\in J$, for all indices $k$. Now, by passing to a subsequence, we have that for every closed ball $B\subset\R^d$
	\begin{equation}\label{eq:74}
	\upphi_k(t)\to\upphi(t)\qquad\text{strongly in }C(B) \text{ for a.a. }t\in(0,T).
	\end{equation}
	The functions $\upphi_k$ are solution of the weak equation
	\begin{align}\label{eq:67}
	&\int_{{\mathbb R^d}}(\varepsilon_0+\varepsilon_1|\nabla\upphi_k(t)|^{p-2})\nabla\upphi_k(t)\cdot\nabla\zeta\,\d \mathsf x
	=\int_{\R^d}\mathsf{\mono}_{\rm ext}\zeta+\int_\varOmega q \zeta\circ\bm\chi_k(t)\,\d \mathsf x
	\end{align}
	for all $\zeta\in W^{1,p}(\R^d)$. Since the right-hand side converges strongly in $W^{1,p}(\R^d)^*$, and since the left-hand side has a uniformly convex potential, it is standard to conclude that $\upphi_k$ converges strongly to $\upphi$ and that $\upphi$ solves
	\begin{align}\label{eq:77}
	&\int_{{\mathbb R^d}}(\varepsilon_0+\varepsilon_1|\nabla\upphi(t)|^{p-2})\nabla\upphi(t)\cdot\nabla\zeta\,\d \mathsf x
	=\int_{\R^d}\mathsf{\mono}_{\rm ext}\zeta+\int_\varOmega q \zeta\circ\bm\chi(t)\,\d \mathsf x.
	\end{align}
	By Fubini's theorem, we can further integrate with respect to $t$ to obtain \eqref{w-sln3}.

	The scalar--valued functions 
	$t\mapsto\|\nablait^2\bm\chi_k(t)-\nablait^2\bm\chi(t)\|_{L^2(\varOmega;\R^{d\times d\times d})}$,
	$t\mapsto \|\nablait\bm\chi_k^{-\top}(t)-\nablait\bm\chi^{-\top}(t)\|_{C(\overline\varOmega;\R^{d\times d})}$, and $t\mapsto \|\nablait\bm\chi_k^{-\top}(t)-\nablait\bm\chi^{-\top}(t)\|_{C(\overline\varOmega;\R^{d\times d})}$ converge to zero strongly in $L^r(I)$, and hence they also converge almost everywhere in $I$, that is,
	\begin{subequations}
		\begin{alignat}{3}
		&\nablait^2\bm\chi_k(t)\to\nablait^2\bm\chi(t)&&\quad\text{ strongly in }L^2(\varOmega;\mathbb R^{d\times d\times d})&&\quad\text{ for a.a. }t\in I,\\
		&\nablait\bm\chi_k(t)\to\nablait\bm\chi(t)&&\quad\text{ strongly in }C(\overline\varOmega;\mathbb R^{d\times d})&&\quad\text{ for a.a. }t\in I,\\
		&\nablait\bm\chi_k^{-\top}(t)\to\nablait\bm\chi^{-\top}(t)&&\quad\text{ strongly in }C(\overline\varOmega;\mathbb R^{d\times d})&&\quad\text{ for a.a. }t\in I,\\
		&\nablait\bm\chi_k(t)\to\nablait\bm\chi(t)&&\quad\text{ strongly in }C(\varGamma;\mathbb R^{d\times d})&&\quad\text{ for a.a. }t\in I.
		\end{alignat}
	\end{subequations}
	We
	now examine the limit passage in the various terms in the approximate form 
	of \eqref{w-sln1}. The limit passage in the first and third terms of the 
	left--hand side of \eqref{w-sln1} are immediate since the deformation appears 
	linearly, and weak convergence suffices. The limit passage in the second term 
	on the left--hand side of \eqref{w-sln1} follows from the continuity 
	properties $\varphi'$ that ensue from the qualification \eqref{cont} of 
	$\varphi$, from the convergence \eqref{eq:63}, and from Lebesgue 
	dominated-convergence theorem. {Here it is important that, even 
		on the Galerkin level, due to} \cite[Theorem 3.1]{HeaKro09IWSS},
	{we are uniformly on a sufficiently big level set of the energy 
		$\int_{\varOmega}\varphi(\nablait\bm\chi)\,\d x+\mathscr{H}(\nablait^2\bm\chi)$ 
		so that we are also uniformly away of the singularity of $\varphi$ at 
		$\det\bm F=0$, so $\varphi'({\bm\chi}_k)$ is even uniformly bounded.}
	
	We now focus our attention on the right--hand 
	side of \eqref{w-sln1}. For the first term, we have
	\begin{subequations}\label{weak-conv}
		\begin{align}\label{weak-conv1}
		&\!\int_\varOmega\!\upphi_k(t){\circ}\bm\chi_k(t)\nablait \bm\chi_k^{-\top}(t):\nablait(q\bm\zeta(t))\,\d x
		\to\!\int_\varOmega\!\upphi(t){\circ}\bm\chi(t)\nablait \bm\chi^{-\top}(t):\nablait(q\bm\zeta(t))\,\d x\,.
		\end{align}
		For the second term, we observe that, since $\nabla\bm\chi^{-T}=\big[\frac{\Cof}{\det}\big](\nabla\bm\chi)$, we have, 
		$$\operatorname{div}(\nabla\bm\chi^{-\top})=\bigg[\frac{\Cof}{\det}\bigg]'(\nablait \bm\chi)\Vdots\nablait^2\bm\chi,$$
		and hence, by the strong convergence of $\upphi_k(t)\circ\bm\chi_k(t)$ and $\nablait\bm\chi_k(t)$, and the weak convergence of $\nablait^2\bm\chi$, we have 
		\begin{align}
		&\hspace*{-3em}\int_\varOmega q \upphi_k(t){\circ}\bm\chi_k(t)\operatorname{div}(\nablait\bm\chi_k^{-\top}(t))\cdot\bm\zeta(t)\,\d x
		\nonumber
		\\[-.4em]
		=&\int_\varOmega q\upphi_k(t){\circ}\bm\chi_k(t)\bigg[\frac{\Cof}{\det}\bigg]'(\nablait \bm\chi_k)\Vdots\nablait^2\bm\chi_k(t)\cdot\bm\zeta(t)\,\d x\nonumber
		\\[-.5em]
		&\qquad\to
		\int_\Omega q\upphi(t){\circ}\bm\chi(t))\bm\zeta(t)\cdot\bigg[\frac{\Cof}{\det}\bigg]'(\nablait \bm\chi(t))
		\Vdots\nablait^2\bm\chi_k(t)\,\d x
		\nonumber\\[-.5em]&\qquad\qquad=
		\int_\varOmega q\upphi(t){\circ}\bm\chi(t)\operatorname{div}(\nablait\bm\chi^{-\top}(t))\cdot\bm\zeta(t)\,\d x.
		\end{align}
		Finally,
		\begin{align}                                                      
		\label{weak-conv3}
		\!\int_\Gamma\!q\upphi_k(t){\circ}\bm\chi_k(t)\nablait \bm\chi_k^{-\top}(t)\bm n\cdot\bm\zeta\,\d S
		\to\!\int_\Gamma\!q\upphi(t){\circ}\bm\chi(t)\nablait \bm\chi^{-\top}(t)\bm n\cdot\bm\zeta\,\d S
		\end{align}\end{subequations}
	for a.a.\ $t\in I$. Note that, in \eq{weak-conv1}, 
	we used that $\nablait\mono\in L^1(\Omega;\R^d)$ while, in
	\eq{weak-conv3}, we used that $\mono|_\Gamma^{}\in L^1(\Sigma)$
	if $\mono\in W^{1,1}(\Omega)$, as assumed. Then we can integrate 
	\eq{weak-conv} over $I$, and prove the convergence again
	by using Lebesgue dominated-convergence theorem, relying on a common majorant \COMMENT{This argument should be slightly expanded.}
	which is even constant due to the $L^\infty(I)$-estimates at disposal.
	The convergence in the approximate form of \eq{w-sln1} towards its limit 
	is then proved.
\end{proof}

\section{Flow of a charged diffusant
	%Diffusion of charges
	$\mono$
	%by Cahn-Hilliard system
}\label{sec-diffusion}
%        ~~~~~~~~~~~~~~~~~~~~~~~~~~~~~~~~~~~~~~~~~~~~

Considering the charge density not fixed with the elastic material
\color{black}but rather with a diffusant that can move throughout
the poroelastic medium \color{black}
(so that the field $\mono$ becomes an additional unknown) and 
then the stored energy $\varphi=\varphi(\nablait\bm\chi,\mono)$, we 
can consider $\calE=\calE(\bm\chi,\mono,\upphi)$
again defined by \eqref{static-gravity-functional+}, i.e.\ 
%but now enhanced by the gradient ``capillarity'' term $\frac12\kappa|\nabla\mono|^2$. Altogether, the enhanced energy is
\begin{align}\nonumber
\calE(\bm\chi,\upphi,\mono)=&
\int_\varOmega\!\varphi(\nablait\bm\chi,\mono)+\mono\upphi\circ\bm\chi
%+\frac12\kappa|\nabla\mono|^2
\,\d x
\\[-.4em]&\qquad
+\int_{\R^d}\mathsf{\mono}_{\rm ext}\upphi
-\frac {\varepsilon_0}2 \big|\nabla\upphi\big|^2
-\frac {\varepsilon_1}p\big|\nabla\upphi\big|^p\,\d\mathsf x
+\mathscr{H}(\nablait^2\bm\chi)\,.
%-\int_\varGamma\bm g\cdot\bm\chi\d S
\label{enhanced-functional+}
\end{align}
Then $\mu:=\partial_\mono\calE(\bm\chi,\mono,\upphi)$
is in the position of an
%(electro)chemical potential (the equivalent of pressure in standard poroelasticity).
{\it electrochemical potential}.
In fact, we 
admit $\varphi$ nonsmooth at $\mono=0$, we rather write it as an inclusion, i.e.
%\COMMENT{THE SIGN IN $\upphi$ TO CHECK STILL}
\begin{subequations}\label{diffusion}
	\begin{align}\label{chem-pot}
	&\mu%:=\partial_\mono\calE(\bm\chi,\mono,\upphi)
	%+\tau_{_{\rm R}}\DT\mono
	\in\partial_\mono\varphi(\nablait\bm\chi,\mono)+\upphi(\bm\chi)\,.
	%-{\rm div}(\kappa\nablait\mono)
	%+\tau_{_{\rm R}}\DT\mono
	\intertext{%is in the position of (electro)chemical potential.
		% augmented by the rate term with $\tau_{_{\rm R}}\ge0$ a relaxation time. 
		The system can now be expanded also by the 
		mass-balance equation written in the reference configuration together
		with the boundary conditions}
	&\label{diffusion+}
	\DT\mono-{\rm div}(\bm M(x,\nablait\bm\chi,\mono)\nablait\mu)=0\qquad\ \ \ \ \ \, 
	\text{ in }\ \varOmega \ \text{ (at a given time $t$)},
	\\
	&%\bm M(x,\nablait\bm\chi)\nablait\mu
	(\bm M(x,\nablait\bm\chi,\mono)\nablait\mu)\cdot n+\alpha\mu=\alpha\mu_\flat\quad
	%\text{ and }\quad(\kappa\nablait\mono)\cdot n=0
	\quad\text{ on }\ \varGamma \ \text{ (at a given time $t$)},
	\end{align}
	where $\bm M=\bm M(x,\nablait\bm\chi,\mono)$ is the mobility tensor, 
	%(meaning the diffusivity),
	% or the electric conductivity tensor
	$\alpha$ is a permeability coefficient of the boundary, $n$ denotes the unit 
	outward normal to the boundary $\varGamma=\partial\varOmega$,
	and $\mu_\flat$ is
	a prescribed external electro-chemical potential. 
	%The system (\ref{diffusion}a,b) is referred as the Cahn-Hilliard
	%equations \cite{CahnH1958JCP}. %\COMMENT{HERE SOME REFS}
	Under the assumption that 
	the mobility tensor in the current configuration, namely 
	$\mathsf M:\Omega\to\R^{d\times d}$ does not depend on $\bm F$, its pullback is
	%\COMMENT{Pushback does not exist: either pushforward or pullback!}
	\begin{align}\label{M-pullback}
	\bm M(x,\bm F,\mono)=\frac{(\operatorname{Cof}\bm F)^\top\bm{\mathsf M}(x,\mono)\operatorname{Cof}\bm F}{\det\bm F}\qquad\text{ with }\ x\in\varOmega,
	\end{align}\end{subequations}
where $\operatorname{Cof} \bm F=(\det \bm F)\bm F^{-\top}$ is the cofactor of 
$\bm F$. Here $\bm{\mathsf M}=\bm{\mathsf M}(x,\mono)$ is the 
material mobility tensor.
In literature, this formula is often used in
%Actually, a reasonable sense of this formula is rather for 
the isotropic case $\bm{\mathsf M}(x,\mono)=m\mono{\mathbb I}$ where 
\eqref{M-pullback} can easily be written by using the right Cauchy-Green 
tensor $C=F^\top F$ as $\det(C^{1/2})k(\mono)C^{-1}=(\det(F^\top)\det(F))^{1/2}
m(\mono)(F^\top F)^{-1}=\det(F)F^{-1}(k(\mono){\mathbb I})F^{-\top}$,
cf.\ e.g.\ \cite[Formula (67)]{DuSoFi10TSMF} or
\cite[Formula (3.19)]{GovSim93CSD2}. For a general case, see 
\cite[Sect.\,9.1]{KruRou18MMCM}.

%Referring to \eqref{options}, $M$ can mean 
%either (a matrix of) diffusivity
% or electric conductivity coefficients. 
%If $\tau_{_{\rm R}}>0$, \eqref{diffusion} is called the viscous Allen-Cahn 
%equation, while $\tau_{_{\rm R}}>0$ is just usual Allen-Cahn equation. 
%If also a ``capillarity-like'' 
%gradient term $\frac12\kappa_0|\nabla\mono|^2$ would 
%be involved so that the first integral in \eqref{static-gravity-functional+} would take the form
%$\int_\varOmega\!\varphi(\nablait\bm\chi,\mono)+\frac12\kappa_0|\nabla\mono|^2\,\d x$, then
%\eqref{chem-pot} would involve the term $-{\rm div}(\kappa_0\nabla\mono)$ and 
%the system \eqref{diffusion} is called (viscous) Cahn-Hilliard equation.
Such system \eqref{system} with \eqref{diffusion} is no longer conservative as 
diffusive processes are dissipative. The dissipation rate is 
$2{\mathscr R}(\bm\chi,\mono,\mu)=\int_\varOmega(\nabla\mu)^\top\bm M(x,\nablait\bm\chi,\mono)
\nabla\mu
%+\tau_{_{\rm R}}\DT\mono^2
\,\d x
+\int_{\partial\varOmega}\alpha\mu^2\,\d S$.
% provided \eqref{diffusion+} is 
%completed by the boundary condition 
%$({\mathsf M}\nablait\mono)\cdot n+\alpha\mono=0$ with $n$ denoting the unit 
%outward normal to the boundary $\partial\varOmega$.
Considering for a moment $\mu_\flat=0$,
analogous to \eqref{abstract-system}, the enhanced system has the abstract 
structure
\begin{subequations}\label{abstract-system+}
	\begin{align}\label{abstract-system+1}
	\partial_{\DT{\bm\chi}}{\mathscr T}
	%(\varrho)
	\DDT{\bm\chi}+\partial_{\bm\chi}\calE(\bm\chi,\mono,\upphi)=\calF(t),
	\\\label{abstract-system+2}
	\partial_{\DT\mono}{\mathscr R}^*(\bm\chi,\mono)\DT\mono+\partial_{\mono}\calE(\bm\chi,\mono,\upphi)=0,
	%A_{\bm\chi(t),\mono}\alpha\mu_\flat(t),
	\\\label{abstract-system+3}
	\partial_{\upphi}\calE(\bm\chi,\mono,\upphi)=0,
	\end{align}\end{subequations}
where the potential of dissipative forces expressed in terms of the rate 
$\DT\mono$, i.e.\
\begin{align}\label{dissip-pot}
{\mathscr R}(\bm\chi,\mono,\DT\mono)=
\int_\varOmega\frac12|\bm M(\nablait\bm\chi,\mono)^{1/2}
\nablait\Delta_{\bm M(\nablait\bm\chi,\mono),\alpha}^{-1}\DT\mono|^2
%+\frac12\tau_{_{\rm R}}\DT\mono^2
\,\d x+\int_{\partial\varOmega}\frac\alpha2(\Delta_{\bm M(\nablait\bm\chi,\mono),\alpha}^{-1}\DT\mono)^2\,\d S\,,
\end{align} 
is quadratic in term of the 
rate $\DT\mono$ with $\Delta_{\bm M,\alpha}^{-1}:\xi\mapsto\mono$ denoting the 
linear operator which assigns $\xi$ the (weak) solution $\mono$ to the 
boundary-value problem $\xi=-{\rm div}(\bm M\nablait\mono)$
with the boundary condition $(\bm M\nablait\mono)\cdot n+\alpha\mono=0$. 
In fact, ${\mathscr R}^*(\bm\chi,\cdot)$ is the convex conjugate functional
to the ${\mathscr R}(\bm\chi,\cdot)$.
% and $A_{\bm\chi,\mono}$ in 
%\eq{abstract-system+2} is a linear operator,\COMMENT{DETAILS} cf.\ also 
A general $\mu_\flat$ would give rise to a non-zero 
right-hand side in \eq{abstract-system+2}, cf.\ also \cite{Roub17ECTD}.

%Note also that the kinetic energy ${\mathscr T}={\mathscr T}(\varrho,\DT{\bm\chi})$ is now to be considered explicitly 
%dependent on $\varrho$ for the case that $\varrho$ depends on 
%$\mono$ like e.g.\ in Remark~\ref{rem-Biot}. 
The energy balance
on a time interval $[0,t]$ can be revealed by testing the particular equations 
\eqref{abstract-system+} respectively by $\DT{\bm\chi}$, $\DT\mono$, and 
$\upphi$:
% with a by-part integration of the kinetic-energy term:
\begin{align}\nonumber
&\!\!\!\!\ddd{{\mathscr T}(
	%\varrho(t),
	\DT{\bm\chi}(t))
	+\calE(\bm\chi(t),\mono(t),\upphi(t))_{_{_{_{_{_{_{}}}}}}}}
{kinetic and stored energy}{at time $t$}\!\!\!\!
+\!\!\!\!\!\ddd{2\int_0^t\!{\mathscr R}(\bm\chi,\mu)\,\d t}{energy dissipated on}{the time interval $[0,t]$}
\\&=\!\!
\ddd{{\mathscr T}(\varrho(0),\DT{\bm\chi}(0))+\calE(\bm\chi(0),\mono(0),\upphi(0))}{kinetic and stored energy}{at the initial time $0$}\!\!\!
%+\frac12\int_0^t\big[\partial_{\varrho\DT{\bm\chi}}^2{\mathscr T}(\varrho)\big]
%(\DT\varrho,\DT{\bm\chi}){\cdot}\DT{\bm\chi}
+\!\!\!\!\!\ddd{\int_0^t\Big(\langle\calF,(\DT{\bm\chi},\DT\upphi)\rangle+\!\int_{\varGamma}\mu_{\rm ext}\mu\,\d S\Big)\,\d t}{work done by external}
{mechano-chemical loading}\!\!\!
\label{energy-balance}
\\[-2.5em]\nonumber
\end{align}
with $\mu=\partial_\mono\calE(\bm\chi,\mono,\upphi)$,
cf.\ \eqref{abstract-system} where ${\mathscr R}=0$.
%where the last term disappears when $\varrho$ does not vary in time, as 
%it is in the conservative case \eqref{abstract-system} where also 
%${\mathscr R}=0$.

A more specific example of free energy is the celebrated Biot model 
\cite{Biot41GTTS} formulated (analogously as at small strains) here at large 
strains, cf.\ also e.g.\ \cite{DuSoFi10TSMF}, resulting into a potential %\COMMENTGT{REFERENCE???}
\begin{align}\nonumber\\[-2.5em]\label{Biot-model}
\varphi
%_{_{\rm BIOT}}
(\bm F,\mono)=\varphi_{_{\rm S}}(\bm F)+
\frac12M_{_{\rm B}}\big(\mono-\mono_\text{\sc e}\beta(1{-}\det\bm F)\big)^2+\begin{cases}
\color{black}\kappa\color{black}
\mono(\ln(\mono/\mono_\text{\sc e})-1)&\text{ for }\ \mono>0,
\\\qquad0&\text{ for }\ \mono=0,
\\\qquad+\infty&\text{ for }\ \mono<0,\end{cases}
\end{align}
where $\mono$ is a mass concentration of a diffusant and $\mono_\text{\sc e}$
is an equilibrium concentration, 
$M_{_{\rm B}}>0$ the so-called Biot modulus,
%and
$\beta\ge0$ the Biot coefficient,
\color{black}and $\kappa>0$ a coefficient\color{black}.
Let us notice the singularity of $\partial_\mono\varphi
%_{_{\rm BIOT}}
(\bm F,\mono)$ at $\mono=0$, which ensures non-negativity of the diffusant
concentration $\mono$. Here $\varphi_{_{\rm S}}$ plays the role of $\varphi$ 
in Sections~\ref{sec-mono} and \ref{sec-mono-anal}.
%\COMMENT{I AM NOT SURE HOW THE SINGULARITY FOR $\mono=0$ CAN BE HANDLED}

To facilitate the analysis, we neglect the mass density of the diffusant, 
cf.\ Remark~\ref{rem-inertia} below.

Let us now briefly present the analysis if the diffusion \eq{diffusion}
is involved. The definition of the weak solution must be formulated 
carefully if $\varrho$ has a singularity as in \eq{Biot-model},
combining the concept of the variational solution for \eq{chem-pot} 
and the conventional weak solution for (\ref{diffusion}b,c)
with the initial condition for $\mono$:

\begin{definition}[Weak solution to the problem \eq{system} with diffusion
	\eq{diffusion}]\label{def2}
	%\COMMENT{TO BE STILL CHECKED}
	A quad\-ruple $(\bm\chi,\upphi,\mono,\mu)\in 
	C_{\rm w}(I;H^{2+\gamma}(\varOmega;\R^d))
	\times L^\infty(I;W^{1,p}(\mathbb R^d))\times C_{\rm w}(I;L^2(\Omega))\times 
	L^2(I;H^1(\Omega))$ is a weak solution to \eq{system} and \eq{diffusion}
	with\eq{BC}--\eq{IC} and the initial condition $\mono|_{t=0}=\mono_0$
	if $\mathscr{H}(\nablait\bm\chi)\in L^\infty(I)$, 
	$\DT{\bm\chi}\in C_{\rm w}(I;L^2(\varOmega;\R^d))$,
	and $\DT\mono\in L^2(I;H^1(\Omega)^*)$, and if 
	%the following conditions hold:\medskip
	\eq{w-sln1} with $\partial_{\bm F}\varphi(\nablait\bm \chi,\mono)$
	instead of $\partial\varphi(\nablait\bm \chi)$ holds together
	with the initial condition $\DT{\bm\zeta}(0)=\bm v_0$,
	%\noindent 1) for every $\bm\zeta\in C^\infty(\overline{Q};\mathbb R^d)$ 
	%satisfying $\bm\zeta(T)=0$,
	%\begin{subequations}\label{w-sln+}\begin{align}\nonumber
	%&\int_Q\partial_{\bm F}\varphi(\nablait\bm \chi,\mono){:}\nablait\bm\zeta
	%%-\varrho\DT{\bm\chi}{\cdot}\DT{\bm\zeta}
	%-\DT\varrho\DT{\bm\chi}{\cdot}\bm\zeta
	%+\bigg(\int_\varOmega(\nablait^2\bm\chi(x')-\nablait^2\bm\chi(x))\vdots\mathscr K(x-x')\,\d x'\bigg)\vdots\nablait^2\bm\zeta(x)\,\d x\d t
	%%+\int_Q(\nablait^2\bm\chi(x')-\nablait^2\bm\chi(x))\vdots\mathscr K(x-x')\vdots\nablait^2\bm\zeta(x)\,\d x\d x
	%\\[-.4em]&\qquad\qquad\qquad\qquad\qquad
	%+\int_Qsm(\nabla\upphi\circ\bm\chi)\cdot\bm\zeta\,\d x\d t
	%=\int_{\varSigma}\bm g{\cdot}\bm\zeta\,\d S\d t
	%-\int_{\varOmega}\varrho\bm\chi_0{\cdot}\DT{\bm\zeta}(0)\,\d x;
	%\label{w-sln1+}\end{align}
	%and also the initial condition $\DT{\bm\zeta}(0)=\bm v_0$ holds.
	and if,
	%\medskip\noindent 2) 
	for every $w\in H^1(Q;\mathbb R^d)$ satisfying $w(T)=0$, 
	it holds
	\begin{align}
	&\int_Q\bm M(\nablait\bm\chi)\nablait\mu{\cdot}\nablait w
	-\mono\DT w\,\d x\d t=\int_\varOmega\mono_0 w(0)\,\d x
	\end{align}
	%\begin{align}\int_Q
	%\big(\partial_{\mono}\varphi(\nablait\bm\chi,\mono)-s\upphi(\bm\chi)\big)w
	%+\tau_{_{\rm R}}\DT\mono w
	%+\tau_{_{\rm R}1}|\nablait\DT\mono|^{p-2}\nablait\DT\mono{\cdot}\nablait w
	%\,\d x\d t=0
	%%\int_{\varOmega} \tau_{_{\rm R}}\mono_0\bm w(0)+\tau_{_{\rm R}1}....\,\d x;
	%\end{align}
	%\end{subequations}
	with $\mu$ satisfying the variational inequality 
	\begin{align}\label{var-ineq}
	\int_\varOmega\varphi(\nablait\bm\chi(t),\mono(t))\,\d x
	\le\int_\varOmega\varphi(\nablait\bm\chi(t),\widetilde\mono)
	+(\mu(t)-\upphi(t,\bm\chi(t)))(\widetilde\mono-\mono(t))\,\d x
	\end{align}
	for a.a.\ $t\in I$ and for all $\widetilde\mono\in L^2(\varOmega)$,
	%from \eq{chem-pot} which is to be satisfied a.e.\ on $Q$, 
	and if also \eq{w-sln3} holds.
\end{definition}

Beside \eqref{nested-X}, 
we now use the Galerkin approximation also for \eqref{diffusion} with 
finite-dimensional subspaces of $H^1(\varOmega)$. It is important (and
we can assume it without loss of generality) that 
both \eqref{chem-pot} and \eqref{diffusion+} use the same finite-dimensional
subspaces of $H^1(\varOmega)$, which facilitates the cross-test 
of  \eqref{chem-pot} by $\DT\mono$ and of \eqref{diffusion+} by $\mu$.

Again, we may assume that the initial
conditions $\bm\chi_0$, $\bm v_0$, and $\mono_0$ 
belong to all these subspaces without loss of generality. 
Let us denote the approximate 
solution thus created by $(\bm\chi_k,\mono_k,\mu_k,\upphi_k)$
It is important to use the same subspace for discretisation of both equations 
in \eqref{diffusion} to allow the cross-test of \eqref{chem-pot} by 
$\DT\mono_k$ and of \eqref{diffusion+} by $\mu_k$. 

\begin{proposition}[Weak solution to \eqref{system} with \eqref{diffusion}]
	Let %$\tau_{_{\rm R}}\ge0$
	$\varphi:{\rm SL}^+(d)\times\R\to\R$ be smooth, 
	$\varphi(\cdot,\mono):{\rm SL}^+(d)\to\R$ satisfy \eq{ass-1}
	uniformly for $\mono\in\R$, $\varphi(\bm F,\cdot)$ be uniformly convex 
	(uniformly also with respect to
	$\bm F$) with
	$\partial_{\bm F\mono}^2\varphi:{\rm GL}^+(d)\times\R\to\R^{d\times d}\times\R$ 
	and $1/\partial_{\mono\mono}^2\varphi:{\rm GL}^+(d)\times\R\to\R^{d\times d}\times\R^+$
	be continuous when extended by zero for $\mono<0$, and
	$|\partial_{\bm F\mono}^2\varphi(\bm F,\mono)|\le C(\bm F)(1+|\mono|)$ with some 
	$C:\R^{d\times d}\to\R^+$ continuous. Let also \eq{ass-kernel-K+}
	and \eq{eq:56-f-g} holds.
	Moreover, let $\bm\chi_0\in H^{2+\gamma}(\varOmega;\R^d)$ and
	$\bm v_0\in L^2(\varOmega;\R^d)$ and $\mono_0\in L^2(\varOmega)$
	with $\varphi(\nablait\bm\chi_0,\mono_0)\in L^1(\varOmega)$ 
	Then, together with the a-priori estimates \eq{est-1}, also 
	%options: 1) $\varrho$ time independent and  $\tau_{_{\rm R}1}\ge0$, or 
	%2) $\varrho$ time dependent, $\tau_{_{\rm R}1}>0$, $p>3$, and $|\partial_\mono\%varphi(\bm F,\mono)|\le C(1+|\mono|)$. Then ....
	\begin{subequations}\label{est-1+}
		\begin{align}%\label{est-1+}
		%&\|\bm\chi_k\|_{W^{1,\infty}(I;L^2(\varOmega;\R^d))\,\cap\,L^\infty(I;H^{2+\gamma}(\varOmega;\R^d))}^{}\le C
		%\ \ \text{ with }\ \ \Big\|\frac1{\det(\nablait\bm\chi)}\Big\|_{L^\infty(Q)}^{}\le C,\\
		&\|\mono_k\|_{L^2(I;W^{1,1}(\varOmega))\,\cap\,L^\infty(I;L^2(\varOmega))\,\cap\,H^1(I;H^1(\varOmega)^*)}^{}\le C\ \ \text{ and }\ \
		\\& 
		\|\mu_k\|_{L^2(I;H^1(\varOmega))}^{}\le C.
		%\\&\|\upphi_k\|_{L^\infty(I;H^1(\R^d))}^{}\le C.
		\end{align}
	\end{subequations}
	There is a subsequence of $\{(\bm\chi_k,\mono_k,\mu_k,\upphi_k)\}_{k\in\N}$ 
	converging weakly* in the topologies indicated in \eqref{est-1} and 
	the limit of any such subsequence is the weak solution to 
	\eqref{system} with \eqref{diffusion} with the initial conditions 
	$\bm\chi(0)=\bm\chi_0$, $\DT{\bm\chi}=\bm v_0$, and $\mono(0)=\mono_0$.
\end{proposition}

\begin{proof}
	%\COMMENT{ careful for the ``dual'' estimate of $\DDT{\bm\chi}_k$ if $\varrho_k$ varies in time -- THIS MIGHT BE A PROBLEM -- SO MAYBE $\varrho$ FIXED??}
	%
	We perform the energetic test of the system \eq{system} with 
	\eq{abstract-system+} by successively $\DT{\bm\chi}_k$, $\DT\phi_k$, 
	$\DT\mono_k$, and $\mu_k$.
	
	By a manipulation as in \eq{eq:61}--\eq{end-of-est} extended now by the 
	dissipative diffusion, we obtain 
	\begin{align}\nonumber
	&
	%\frac{\d}{\d t}\bigg(\int_\Omega\frac\varrho2|\DT{\bm\chi}|^2+\varphi(\nabla \bm\chi)
	%%+\mono(\phi\circ \bm\chi)
	%\,\d x+\int_{\R^d}\frac{\eps_0}2\big|\nabla\phi(z)\big|^2
	%+\frac{\eps_1}{p'}|\nabla\phi(z)\big|^p\,\d z+\mathscr{H}(\nabla^2\bm\chi)\bigg)
	%\\&\nonumber=
	\frac{\d}{\d t}\bigg(
	\int_\varOmega\frac\varrho2|\DT{\bm\chi}_k|^2+\varphi(\nablait\bm\chi_k,\mono_k)\,\d x
	+\int_{\R^d}\frac{\eps_0}2\big|\nabla\upphi_k\big|^2
	+\frac{\eps_1}{p'}|\nabla\upphi_k\big|^p\,\d \mathsf x+\mathscr{H}(\nabla^2\bm\chi_k)\bigg)
	\\[-.2em]&\nonumber\qquad\qquad
	+\int_\varOmega\bm M(\nablait\bm\chi_k)\nablait\mu_k{\cdot}\nablait\mu_k\,\d x
	+\int_\varGamma\alpha\mu_k^2\,\d S
	%\UUU{=}{\eq{system} tested}{by $\DT y_k$ and $\DT\phi_k$}
	\\[-.2em]&\qquad\qquad\qquad\qquad=
	\int_\varOmega\bm f\cdot\DT{\bm\chi}_k\,\d x+\int_\Gamma\bm g\cdot\DT{\bm\chi}_k
	+\alpha\mu_\flat\mu\,\d S
	+\int_{\R^d}\DT{\mathsf{\mono}}_{\rm ext}\upphi_k\,\d \mathsf x\,.
	\label{system-est+}
	\end{align}
	From this \COMMENT{while using the assumption that $\DT{\mathsf{\mono}}_{\rm ext}=0$}
	and also by comparison of $\DT\mono_k$, we obtain the a-priori 
	estimates \eq{est-1} and,
	except the $L^2(H^1)$-estimate of $\mono_k$, also 
	\eq{est-1+}. More in detail, the $L^2(I;H^1(\varOmega)^*)$-estimate 
	of $\DT\mono_k$ in \eq{est-1+} is meant for a suitable Hahn-Banach extension
	of the functional $\DT\mono_k$ defined originally only on the 
	functions valued on the finite-dimensional subspace used for the $k$-th
	Galerkin approximation.
	
	As $\mono$ occurs nonlinearly in the energy $\varphi$, we need to prove
	its strong convergence. 
	Even more, the limit passage in \eq{weak-conv} modified for 
	$\mono_k\to\mono$ needs the strong convergence of 
	%$\nablait\mono_k(t)$ in $L^2(I;H^1(\varOmega;\R^d))$ 
	$\mono_k(t)$ in $W^{1,1}(\varOmega)$ because, from a mere weak convergence
	of $\nablait\mono_k$, we could not inherit the convergence for a.a.\ $t$
	needed in \eq{weak-conv}. As we do not use the Cahn-Hilliard model,
	we do not have $\nablait\mono$ directly estimated,
	% in contrast to Sect.\,\ref{sec-dip}, 
	but we can rely on the uniform convexity of $\varphi(\bm F,\cdot)$, as 
	indeed e.g.\ in the Biot model \eq{Biot-model}.
	Then the estimate of $\nablait\mono$ can be obtained 
	%even for $\kappa=0$ 
	by applying $\nabla$-operator to \eq{chem-pot} to see 
	$\nablait\mu=\partial_{\bm F\mono}^2\varphi(\bm F,\mono)\nablait\bm F+
	\partial_{\mono\mono}^2\varphi(\bm F,\mono)\nablait\mono$, so that we can estimate
	\begin{align}
	\nablait\mono_k=
	%[\partial_{\mono\mono}^2\varphi(\nablait\bm\chi_k,\mono_k)]^{-1}(
	\frac{\nablait\mu_k
		-\partial_{\bm F\mono}^2\varphi(\nablait\bm\chi_k,\mono_k)\nablait^2\bm\chi_k}{\partial_{\mono\mono}^2\varphi(\nablait\bm\chi_k,\mono_k)}.
	%\epsilon\|\mono_k{-}\mono\|_{L^2(Q)}^2&\le
	%\int_Q\big(\partial_{\mono}\varphi(\nablait\bm\chi_k,\mono_k)-\partial_{\mono}\varphi(\nablait\bm\chi_k,\mono)\big)(\mono_k{-}\mono)\,\d x\d t
	%\\&\nonumber
	%\le\int_\varOmega\tau_{_{\rm R}}|\mono_k(T){-}\mono(T)|^2\,\d x
	%+\int_Q\big(\partial_{\mono}\varphi(\nablait\bm\chi_k,\mono_k)-\partial_{\mono}%\varphi(\nablait\bm\chi_k,\mono)\big)(\mono_k{-}\mono)\,\d x\d t
	%\\&\nonumber=\int_Q(\mu_k-s\upphi_k(\bm\chi_k))(\mono_k{-}\mono)\,\d x\d t
	%\\[-.3em]&\qquad
	%-\int_\varOmega\mono(T)(\mono_k(T){-}\mono)(T)\,\d x
	%-\int_Q\partial_{\mono}\varphi(\nablait\bm\chi_k,\mono)(\mono_k{-}\mono)\,\d x\d t\to0,
	\label{strong-conv-mono}
	\end{align}
	provided we assume still the uniform convexity of $\varphi(\bm F,\cdot)$, as 
	indeed e.g.\ in the Biot model \eq{Biot-model}. This gives
	the $L^2(H^1)$-estimate of $\mono_k$ in \eq{est-1+}.
	
	Now we select the weakly* converging subsequence in the topologies 
	indicated in \eq{est-1} and \eq{est-1+}. Furthermore, we 
	prove $\nablait\mu_k\to\nablait\mu$ strongly in $L^2(Q;\R^d)$. 
	This can be seen from the initial-boundary-value problem 
	(\ref{diffusion}b,c) when taking into account the uniform positive-definiteness
	of the pulled-back mobility matrix \eq{M-pullback}. We have also 
	$\nablait^2\bm\chi_k\to\nablait^2\bm\chi$ strongly in $L^2(Q;\R^{d\times d\times d})$ 
	and $\mono_k\to\mono$ strongly in $L^2(Q)$. 
	%by the Aubin-Lions theorem, 
	%from \eq{strong-conv-mono} we can see that 
	%$\nablait\mono_k\to\nablait\mono$ in $L^1(Q;\R^d)$. 
	In particular, selecting possibly another subsequence, 
	we have $\nablait\mu_k\to\nablait\mu$ strongly in $L^2(\varOmega;\R^d)$ and 
	$\nablait^2\bm\chi_k(t)\to\nablait^2\bm\chi(t)$ strongly in 
	$L^2(\varOmega;\R^{d\times d\times d})$ and also $\mono_k(t)\to\mono(t)$ strongly in 
	$L^2(\varOmega)$, from \eq{strong-conv-mono} we can see that 
	$\nablait\mono_k(t)\to\nablait\mono(t)$ strongly in $L^1(Q;\R^d)$. Here the 
	continuity of $\partial_{\bm F\mono}^2\varphi$ and of 
	$1/\partial_{\mono\mono}^2\varphi$
	has been used.
	
	%where the equality is due to \eqref{chem-pot} in its Galerkin approximation 
	%tested by $\mono_k{-}\mono$.\COMMENT{STILL: here $\mono$ to be approximated}
	%For the convergence to 0 we used that 
	%$\upphi_k(\bm\chi_k)\to\upphi(\bm\chi)$ weakly in $L^2(I;H^1(\varOmega))$ 
	%\COMMENT{HERE WE NEED $\nabla\upphi_k$ ESTIMATED!! OR Cahn-Hilliard model !?}
	%while 
	%............ $\mono_k\to\mono$ 
	%%weakly in $H^1(I;L^2(\varOmega))$ due to $\tau_{_{\rm R}}>0$,
	%%and thus also 
	%strongly in $L^2(I;H^1(\varOmega)^*)$ by the Aubin-Lions theorem,
	%so that $\upphi_k(\bm\chi_k)\mono_k\to\upphi(\bm\chi)\mono$
	%weakly in $L^1(Q)$.
	
	%\hrule
	
	Note also that the arguments \eq{eq:67}
	%now
	modifies because
	the last integral in \eq{eq:67} is now 
	$\int_{\varOmega}\mono_k
	(\upphi_k{\circ}\bm\chi_k)-\mono_j(\upphi_j{\circ}\bm\chi_j))\,\d \mathsf x$
	and it again converges to 0 when $k,j\to\infty$.
	
	The limit passage in 
	\begin{align*}%\label{var-ineq}
	\int_\varOmega\varphi(\nablait\bm\chi_k(t),\mono_k(t))\,\d x
	\le\int_\varOmega\varphi(\nablait\bm\chi_k(t),\widetilde\mono)
	+(\mu_k(t)-\upphi_k(t,\bm\chi_k(t)))(\widetilde\mono-\mono_k(t))\,\d x
	\end{align*}
	towards \eq{var-ineq} is then easy by lower semicontinuity or continuity.
	%
	%\begin{align}\label{eq:67+}
	%&\int_{{\mathbb R^d}}\left((\varepsilon_0+\varepsilon_1|\nabla\upphi_k|^{p-2})\nabla\upphi_k-(\varepsilon_0+\varepsilon_1|\nabla\upphi_j|^{p-2})\nabla\upphi_j\right)(\nabla\upphi_k-\nabla\upphi_j)\,\d \mathsf x
	%\nonumber\\&\qquad=
	%\int_{\R^d}\mathsf{\mono}_{\rm ext}(\upphi_k-\upphi_j)\,\d \mathsf x+
	%\int_{\varOmega}\!\!\!\mono_k
	%(\upphi_k{\circ}\bm\chi_k)-\mono_j(\upphi_j{\circ}\bm\chi_j))\,\d \mathsf x.
	%\end{align}
	%
\end{proof}

%Let us note that 

% This would .... but the limit passage in \eq{weak-conv} modified for 
%$\mono_k\to\mono$ needed the strong convergence of $\nablait\mono_k$ 
%in $L^2(I;H^1(\varOmega;\R^d))$ because, from a mere weak convergence
%of $\nablait\mono_k$ we could not inherit the convergence for a.a.\ $t$
%needed in \eq{weak-conv}. 
%\COMMENT{MAYBE strong convergence of $\nablait\mu$ !!!}

\begin{remark}[{\it Maxwell stress}]
	{\rm
		The term $q\nabla\upphi$ on the left-hand side of the pointwise force balance \eqref{system1} is the opposite of the electrostatic body-force referential (Lagrangean) density. Its spatial (Eulerian) density can be written as the divergence of a tensor field, which is then recognized as the Maxwell stress. We illustrate this fact by a formal calculation, performed under the assumption that $\bm\chi$ is invertible and that the external charge density $\mathsf q_{\rm ext}$ vanishes.
		As a start, we test \eqref{system1} by a virtual velocity $\bm\zeta$ which is smooth and vanishes in a neighborhood of $\partial\varOmega$, and then we perform an integration over $\varOmega$. In the resulting equation we focus our attention on the term
		$$
		W(\bm\zeta):=\int_\varOmega q(\nabla\upphi{\circ}\bm\chi)\cdot\bm\zeta\,\dx=\int_\varOmega q(\nabla\upphi\cdot(\bm\zeta{\circ}\bm\chi^{-1})){\circ}\bm\chi \dx,
		$$
		which we interpret as the virtual work performed by the referential electrostatic body force.
		%
		%Then,
		Further, we use the weak form \eqref{w-sln3} of the nonlinear Poisson equation \eqref{system2}, substituting $\nabla\upphi\cdot(\bm\zeta{\circ}\bm\chi^{-1})$ in place of the test field $\zeta$. By doing so we obtain
		\begin{equation*}
		\begin{aligned}
		W(\bm\zeta)&=\int_{\bm\chi(\varOmega)}(\varepsilon_0+\varepsilon_1|\nabla\upphi|^{p-2})\nabla\upphi\cdot\nabla(\nabla\upphi\cdot(\bm\zeta{\circ}\bm\chi^{-1}))\\[-.6em]
		&\qquad\ =\int_{\bm\chi(\varOmega)}(\varepsilon_0+\varepsilon_1|\nabla\upphi|^{p-2})\nabla\upphi\cdot(\nabla\nabla\upphi\,(\bm\zeta{\circ}\bm\chi^{-1})+\nabla(\bm\zeta{\circ}\bm\chi^{-1})^\top\nabla\upphi)
		\end{aligned}
		\end{equation*}
		Then, we recall that in our regularized model the spatial electrostatic energy density is $e(\nabla\upphi)$, where $e(\bm v)=\frac {\varepsilon_0}2|\bm v|^2+\frac {\varepsilon_1}{p}|\bm v|^p$. Thus, on introducing the electric field $\bm{\mathsf e}=-\nabla\upphi$, the electric displacement $\bm{\mathsf d}=e'(\bm{\mathsf e})=(\varepsilon_0+\varepsilon_1|\nabla\upphi|^{p-2})\bm{\mathsf e}$, and the \emph{spatial test velocity} $\bm\upzeta=\bm\zeta\circ\bm\chi^{-1}$, we can write
		\begin{equation*}
		\begin{aligned}
		W(\bm\zeta)
		&=\int_{\bm\chi(\varOmega)}e'(\nabla\bm{\mathsf e})\cdot(\nabla\bm{\mathsf e}\,\bm\upzeta)+\bm{\mathsf d}\cdot\nabla\bm\upzeta^\top\bm{\mathsf e}
		\\[-.4em]
		&\qquad\ =\int_{\bm\chi(\varOmega)}\nabla e(\bm{\mathsf e})\cdot\bm\upzeta+(\bm{\mathsf e}\otimes\bm{\mathsf d})\cdot\nabla\bm\upzeta
		%\\    &
		=\int_{\bm\chi(\varOmega)}((\bm{\mathsf e}\otimes\bm{\mathsf d})-e(\bm{\mathsf e})\bm{\mathsf I}):\nabla\bm\upzeta.
		\end{aligned}
		\end{equation*}
		We recognize the spatial tensor field
		$\bm{\mathsf M}=(\bm{\mathsf e}\otimes\bm{\mathsf d})
		-e(\bm{\mathsf e})\bm{\mathsf I}$ to be the so-called Maxwell stress.
	}
\end{remark}

\color{black}
\begin{remark}[{\it Darcy or Fick diffusion in the Biot poroelastic model}]\label{rem-Darcy-Fick}
	\upshape
	For a special choice of the Biot model \eqref{Biot-model}
	for $\varphi=\varphi(\nablait\bm\chi,\mono)$ in \eqref{enhanced-functional+},
	we obtain the stress and the chemical potential as
	\begin{subequations}\begin{align}%\nonumber\\[-2.5em]
		\label{Biot-model-stress}
		&\stress=\varphi_{_{\rm S}}'(\bm F)+
		\beta\mono_\text{\sc e} p\operatorname{Cof}\bm F\ \ \text{ with }\ \ p=M_{_{\rm B}}
		\big(\mono-\mono_\text{\sc e}\beta(1{-}\det\bm F)\big)
		\ \ \text{ and}
		\\&\label{Biot-model-mu}
		\mu\in\partial_\mono\varphi_{_{\rm S}}(\bm F,\mono)=\begin{cases}
		%M_{_{\rm B}}\big(\mono-\mono_\text{\sc e}\beta(1{-}\det\bm F)\big)
		p+\kappa\ln(\mono/\mono_\text{\sc e})&\text{ for }\ \mono>0\,.\\[.1em]
		\qquad\emptyset&\text{ for }\ \mono=0,\end{cases}
		\end{align}\end{subequations}
	%\marginpar{What are the dimensions of $\beta$ and $q_{\textsc e}$? If the
	%product $\beta q_{\textsc e}$ is not dimensionless, then $p$ does not have
	%dimensions of a pressure. In this case, I would incorporate
	%$q_{\textsc e}\beta$ in the definition of pressure $p$.}
	For a standard choice $\bm{\mathsf M}(\mono)=\mono\bm{\mathsf M}_0$ in
	\eqref{M-pullback}, the flux $j=-\bm M(\bm F,\mono)\nablait\mu$
	results to
	\begin{align}
	j=\!\!\!\!\!\!\ddd{-\mono\bm M_0(\bm F)\nablait p_{_{}}}{Darcy law}
	\!\!\!\!\!\!\!\!\!\!\ddd{-\,\kappa\bm M_0(\bm F)\nablait\mono_{_{}}}{Fick law}
	\ \text{ with }\ \ \bm M_0(\bm F)=\frac{(\operatorname{Cof}\bm F)^\top\bm{\mathsf M}_0\operatorname{Cof}\bm F}{\det\bm F}
	\label{Darcy/Fick}
	\\[-2.5em]\nonumber\end{align}
	provided $\mono>0$. Depending on the coefficient $\kappa>0$, either
	Darcy's mechanism or the Fick's one may dominate. In particular, we identified
	the ``diffusant pressure'' $p$ in \eqref{Biot-model-stress} which governs the
	Darcy law. For the Biot model under large strains, we also refer to
	\cite{biot1972theory} and, at small strains, to \cite{Roub17VMSS}.
	Counting the charges as occurring in the
	electrochemical potential \eqref{chem-pot}, the Darcy/Fick law
	\eqref{Darcy/Fick} enhances still as
	\begin{align}
	j=-\mono\bm M_0(\bm F)\nablait p-\kappa\bm M_0(\bm F)\nablait\mono-\mono\bm M_0(\bm F)\nablait\upphi(\bm\chi)\,,
	%  j=\!\!\!\!\!\!\ddd{-\mono\bm M_0(\bm F)\nablait p_{_{}}}{Darcy law}
	%\!\!\!\!\!\!\!\!\!\!\ddd{-\kappa\bm M_0\nablait\mono_{_{}}}{Fick law}
	%\!\!\!\!\!\!\!\!\!\!\ddd{-\mono\bm M_0(\bm F)\nablait\upphi(\bm\chi)}{drift current}
	\end{align}
	where the last term is a so-called drift current. We thus can see a
	drift-diffusion model (as used e.g.\ in monopolar semiconductors) combined
	with the Darcy flow due to mechanical pressure gradient.
\end{remark}
\color{black}

\section{Final remarks: \color{black}generalizations or modifications\color{black}}\label{sec-rem}
%        ~~~~~~~~~~~~~

Let us end this article by several remarks outlining possible
generalizations or modifications and, in most of cases, serious
difficulties related with them.

\begin{remark}[{\it Multi-component flow in charge poroelastic solid}]
	\upshape
	The models from Sections \ref{sec-mono} and \ref{sec-diffusion}
	can be easily merged in the sense that the fixed $\mono$ from
	Section \ref{sec-mono} is to be interpreted as the charge density
	of
	%depends
	\color{black}dopants \color{black}
	in the poroelastic solid while the varying $\mono$ from
	Section \ref{sec-diffusion} is the charge density
	of diffusant. \color{black}This may model elastic negatively-charged
	porous polymers with hydrogen cations (protons) in
	so-called polymer electrolytes as the central layer
	of PEM fuel cells \cite{ProWet09PEMF} or
	unipolar (elastic) semiconductor devices as field-effect transistors FETs or
	Gunn's diodes. \color{black} Even, this latter $\mono$ may be vector-valued if
	there are more than one diffusant. Even chemical reactions 
	between particular components of the diffusant can be involved.
	Then one should rather speak about concentrations $c_i$, $i=1,...,n$,
	of the charged diffusants with the specific charges $\mono_i$.  
	%Denoting $\mono=(\mono_0,\mono_1,...,\mono_n)$, 
	Denoting the charge of the fixed dopant as $\mono_0$,
	the system \eq{system}--(\ref{diffusion}a,b) can be generalized as:
	\begin{subequations}\label{diffusion-n}
		\begin{align}\label{system1-n}
		&\varrho\DDT{\bm\chi}
		-{\rm div}\big(\varphi'(\nablait\bm\chi)
		-{\rm div}\,\mathfrak{H}(\nablait^2\bm\chi)\big)
		%\,{\stress}
		+\mono\nabla\upphi(\bm\chi)=\bm f\,\ \ \ \
		%\text{ with }\ {\stress}=\varphi'(\nablait\bm\chi)
		%-{\rm div}\,\mathfrak{H}(\nablait^2\bm\chi),
		\\\label{system2-n}
		&\operatorname{div}((\varepsilon_0+\varepsilon_1|\nabla\upphi|^{p-2})
		\nabla\upphi)+\!\!\!\sum_{x\in\stackrel{\leftarrow}{\bm\chi}(\cdot,t)}\!
		\frac{\mono(x)+\sum_{i=1}^nc_i(x)\mono_i(x)}{\det(\nablait\bm\chi(t,x))}
		+\mathsf{\mono}_{\rm ext}(t,\cdot)=0
		\\
		&\label{diffusion+n}
		\DT c_i-{\rm div}(\bm M(x,\nablait\bm\chi,c_1,...,c_n)\nablait\mu)=
		r_i(c_1,...,c_n)
		\qquad
		%\text{ in }\ \varOmega 
		\ \text{ ($i=1,...,n$)},
		\\&\label{chem-pot-n}
		\mu_i%:=\partial_\mono\calE(\bm\chi,\mono,\upphi)
		%+\tau_{_{\rm R}}\DT\mono
		\in\partial_{c_i}\varphi(\nablait\bm\chi,c_1,...,c_n)+\upphi(\bm\chi)
		\qquad
		%\text{ in }\ \varOmega 
		\hspace*{7.2em}\ \text{ ($i=1,...,n$)},
		\end{align}\end{subequations}
	with $r_i(\mono)$ the chemical-reaction rate of the $i$-th constituent
	and $\bm M$ is now valued in $\R^{d\times d\times n}$, caring also about 
	cross-diffusion effects. \color{black}The specific applications may cover  
	hydrogen oxidation and the oxygen reduction reactions in porous electrodes
	(i.e.\ cathode/anode layers) of PEM fuel cells \cite{ProWet09PEMF} influenced by mechanical
	loading \cite{BCMK05DRFC,MaBoBen08VRNE}
	or in bipolar doped semiconductor devices (diodes, transistors, thyristors)
	with chemically reacting (through generation/recombination mechanisms)
	electron and hole charges under mechanical load, or
	Li-cations and electrons in porous batteries, etc.\color{black}
\end{remark}

%\begin{remark}[Actual Cahn-Hilliard model]
%\upshape
%.... the gradient  pulled-back from the actual deformed 
%configuration into the reference one, i.e.\
%$\frac12\kappa|\nablait\bm\chi^{-\top}\nablait\mono|^2$
%This gives rise also to a Korteweg stress in the momentum equilibrium.
%\COMMENT{MAYBE THIS REMARK NOT NEEDED IF Cahn-Hilliard NOT NEEDED?}
%\end{remark}

\begin{remark}[{\it Inertia of the diffusant}]\label{rem-inertia}
	\upshape
	If the mass density of the diffusant denoted here by $\varrho_1$ is not 
	negligible, one should rather consider the overall mass density as 
	$\varrho=\varrho_0+\mono\varrho_1$ with $\varrho_0$ now standing for the 
	mass density of the poroelastic solid itself (e.g.\ a polymeric matrix or a 
	porous rock).
	The energetic test of the inertial term $\varrho\DDT{\bm\chi}$ 
	by $\DT{\bm\chi}$ would then lead to 
	$\frac{\partial}{\partial t}\frac12\varrho|\DT{\bm\chi}|^2
	-\frac12\DT\mono\varrho_1|\DT{\bm\chi}|^2$. The newly arising term 
	$\frac12\DT\mono\varrho_1|\DT{\bm\chi}|^2$, which would have to be estimated 
	``on the right-hand side, seems to bring serious difficulties even for the 
	a-priori estimation.
	%
	%The overall mass is $\varrho(t,x)=\varrho_{_{\rm S}}(x)+\mono(t,x)$ with $\varrho_{_{\rm S}}$ the mass density of 
	%the poroelastic medium itself (e.g.\ a polymeric matrix or a porous rock) with the elastic 
	%response given by the potential $\varphi_{_{\rm S}}$.
	%%Then $\mono=\varrho$, cf.\ \eqref{options}, we can take $\varphi(F,\mono)=\varphi_{_{\rm BIOT}}(F,\mono{-}\varrho_{_{\rm S}})$.
	%Note that $\varrho$ in kinetic energy now depends in time, so that 
	In the weak formulation \eqref{w-sln1}, 
	we would obtain still the term 
	$\int_Q\DT\varrho\DT{\bm\chi}{\cdot}v\,\d x\d t$ %\COMMENT{NOTATION TO MODIFY} 
	which requires to introduce a viscosity
	%$\tau_{_{\rm R}}>0$
	in \eqref{chem-pot} to control 
	$\DT\varrho=(\varrho_{_{\rm S}}{+}\mono)\!\DT{^{}}=\DT\mono$ in $L^2(Q)$.
	%\COMMENT{EVEN MORE: 
	For analysis, \eqref{chem-pot} is to be still augmented
	by the term $-{\rm div}(\tau_{_{\rm R}}|\nablait\DT\mono|^{p-2}\nablait\DT\mono)$
	with $p>3$ and  with $\tau_{_{\rm R}}>0$ a relaxation time; 
	such a gradient-viscous Cahn-Hilliard model was 
	suggested in \cite{Gurt96GGLC}, and then, assuming a suitable regularization of \eqref{Biot-model}
	to guarantee $|\partial_\mono \varphi(\bm F,\mono)|\le C(1{+}|\mono|)$, e.g.\
	$\det\bm F/(1{+}\eps\det\bm F)$ instead of $\det\bm F$ for small $\eps>0$,
	we can test separately \eqref{abstract-system+2} by $\DT\mono$ to estimate
	$\DT\mono$ in $L^p(I;W^{1,p}(\varOmega))$ and then, using 
	$W^{1,p}(\varOmega)\subset L^\infty(\varOmega)$, treat 
	$\int_0^t\int_\varOmega\DT\varrho|\DT{\bm\chi}|^2,\d x\d t$ by Gronwall inequality.
	Then \eqref{dissip-pot} is to be augmented by some viscosity 
	contribution like 
	$\tau_{_{\rm R}}|\nablait\DT\mono|^p/p$.
	%IS IT ALSO SOMEWHERE ELSE IN LITERATURE??} ................. and 
	Also the ``dual'' estimate of $\varrho\DDT{\bm\chi}$
	seems problematic.
\end{remark}

\begin{remark}[{\it Attractive monopolar interactions}]
	\upshape
	Another prominent example of a monopolar interaction is gravitation. Then
	$\mono$ is mass density and $\phi$ is the gravitational field. 
	The essential difference is that the gravitational constant
	occurs instead of $\eps$ in \eq{static-gravity-functional} and then also in 
	\eq{system2} and \eq{gravity-large-strain} with a negative
	sign.
	The coercivity of the static stored energy is not automatic and, roughly 
	speaking needs sufficiently small total mass of a medium sufficiently 
	elastically tough. 
	%It is needed also in 
	%the dynamical case where its weaker coercivity when 
	%$\|\cdot\|_{L^2(\Omega;\R^d)}^2$ is added with a sufficiently large weights suffices.
	On top of it, $\eps_1=0$ is the only reasonable choice in such gravitation
	interaction, which is not covered by the proof of 
	Proposition~\ref{prop-1}.
\end{remark}

\begin{remark}[{\it Dipolar long-range interactions}]\label{rem-dipol-inter}
	\upshape
	Considering a vector-valued density $\vec{\mono}:\Omega\to\R^m$ and
	$\vec{\mono}_{\rm ext}^{}(\cdot,t):\R^d\to\R^m$ instead of just scalar valued does 
	not change \eq{eq:55} and 
	%Lemma~\ref{Giuseppe-lemma}
	\eq{eq:55}, but the stored energy 
	%\eq{static-gravity-functional-evol++}
	\eq{static-gravity-functional+} is to be modified as:
	\begin{align*}
	\nonumber
	\calE(\bm\chi,\phi)=\!
	\int_\Omega\!\varphi(\nabla \bm\chi)-\vec{\mono}\cdot\nabla\phi(\bm\chi)\,\d x
	+\int_{\R^d}\frac\kappa2\big|\nabla\phi(\mathsf x)\big|^2
	%-\vec{\mono}_{\rm ext}\cdot\nabla\phi
	\,\d\mathsf x+\mathscr{H}(\nabla^2\bm\chi)\,.
	%\label{static-gravity-functional+++}
	\end{align*}
	Instead 
	of \eq{system}, the Hamiltonian variational principle then modifies
	the system \eq{system} as
	\begin{subequations}\label{system-dipol}
		\begin{align}\label{system1-dipol}
		&\varrho\DDT{\bm\chi}
		-{\rm div}\,\stress=\bm f+[(\nabla^2\phi)\circ \bm\chi]\vec{\mono}\ \ \ \ \
		\text{ with }\ \ \ 
		\stress=\varphi'(\nabla \bm\chi)-{\rm Div}(\mathfrak{H}\nablait^2\bm\chi)
		%&&\text{on }Q,
		\\\label{system2-dipol}
		&{\rm div}(\varepsilon\nabla\phi)=%\chi_{y(\Omega)}^{}
		{\rm div}\,\bigg(\sum_{x\in \bm\chi^{-1}(\cdot,t)}\frac{\vec{\mono}(x)}
		{\det(\nabla \bm\chi(x,t))}
		+\vec{\mono}_{\rm ext}^{}(\cdot,t)\bigg)\ \ \ \ \text{for a.a.\ $t\in I$}.
		%&&\text{on }\R^d
		\end{align}\end{subequations}
	The interpretation is of $\phi$ and 
	$\vec{\mono}$ and of $\kappa$ is the potential of magnetic field and magnetization and permeability in elastic 
	ferromagnets, respectively, or alternatively electrostatic field and 
	polarization and permittivity 
	in elastic piezoelectric materials with spontaneous polarization.
	The analysis is, however, even more problematic comparing to the repulsive 
	monopolar case due to less regularity  of the Poisson equation 
	\eq{system2-dipol} which has the divergence of an $L^1$-function
	in the right-hand side. Actually, this difficulty is not seen in 
	the a-priori estimation strategy \eq{system-est+} as well as in the limit
	passage in the Poisson equation \eq{system2-dipol} written in the form 
	\eq{w-sln3}.
	%can again be done for smooth test functions.
	Yet, the difficulty occurs in the term $(\nabla^2\phi)\circ \bm\chi$ in the 
	right-hand side of \eq{system1-dipol} because $\nabla^2\phi$ hardly can 
	be continuous and the composition with $\bm\chi$ even does not need to be 
	measurable.
\end{remark}

%\marginpar{WE SHOULD ALSO ADD A REMARK WHERE WE EXPLAIN WHY THE EXTERNAL CHARGE CANNOT BE TIME INDEPENDENT}

%\section*{Acknowledgments}
%This research was partly supported through the grants Czech Science Foundation 
%16-03823S, 
%%``Homogenization and multi-scale computational modeling of flow and 
%%nonlinear interactions in porous smart structures'',
%16-34894L, 
%%``Variational structures in continuum thermomechanics of solids'', 
%and
%17-04301S 
%%``Advanced mathematical methods for dissipative evolutionary systems'',
%as well as through the institutional project RVO:\,61388998 (\v CR).

%GT acknowledges support from the Grant of Excellence Departments, MIUR-Italy (ARTICOLO 1, COMMI 314-337 LEGGE 232/2016), and the support of INdAM-GNFM.

%\COMMENT{MAYBE STILL Vienna - MARTIN TO ASK}

%\COMMENT{?????}
%He also thanks Miroslav \v{S}ilhav\'y for fruitful discussions.

%\COMMENT{SOME PAPERS IN J.Elast. (BUT I DID NOT SEE THEM):\\
%V.K. Kalpakides E.K. Agiasofitou:
%On Material Equations in Second Gradient Electroelasticity.
%Journal of Elasticity (2002) 67: 205--227. 
%https://doi.org/10.1023/A:1024926609083
%\\.\\
%M. Singh P. D. S. Verma:
%Nonlinear couple stress theory of elastic dielectrics with applications to 
%dynamic deformations. Journal of Elasticity 1983, Vol13, pp 379--393 
%\\.\\
%A. Dorfmann R. W. Ogden:
%Nonlinear Electroelastic Deformations.
%Journal of Elasticity 2006, Volume 82, Issue 2, pp 99--127}

\bibliographystyle{abbrv}

\end{sloppypar}

\end{document}